\documentclass[a4paper,reqno]{amsart}
\usepackage{preemble-art}

\usepackage{color}

\begin{document}

\title[Discrete harmonic maps]
      {Uniformizing surfaces via discrete harmonic maps}
     
\author[T. Kajigaya]{Toru Kajigaya}
\author[R. Tanaka]{Ryokichi Tanaka} 
%\fnref{myfootnote}}
\address{Department of mathematics, School of Engineering, Tokyo Denki University,
5 Senju Asahi-cho, Adachi-ku, Tokyo 120-8551, JAPAN\endgraf
National Institute of Advanced Industrial Science and Technology (AIST), MathAM-OIL, Sendai 980-8577 
JAPAN
}
\email{kajigaya@mail.dendai.ac.jp}
\address{Mathematical  Institute,  Tohoku  University,  6-3  Aza-Aoba,
Aramaki,  Aoba-ku,
Sendai 980-8578 JAPAN}

\email{ryokichi.tanaka@tohoku.ac.jp}
\date{\today}

%32G15 (Moduli of Riemann surfaces, Teichm\UTF{00FC}ller theory), 58E20 (Harmonic maps), 05C10 (Planar graphs; geometric and topological aspects of graph theory)

\maketitle

\begin{abstract}
We show that for any closed surface of genus greater than one and for any finite weighted graph filling the surface, there exists a hyperbolic metric which realizes the least Dirichlet energy harmonic embedding of the graph among a fixed homotopy class and all hyperbolic metrics on the surface.
We give explicit examples of such hyperbolic surfaces 
through a new interpretation
of the Nielsen realization problem for the mapping class groups.
\end{abstract}

\section{Introduction}

The Koebe-Andreev-Thurston theorem states that any skeleton of triangulation for an oriented compact surface arises as the contact graph of a circle packing on a geometrized surface (\cite[Section 13.7]{Thurston}; see also \cite{ColindeVerdiere}).
In the case of a closed Riemann surface of genus greater than one, one obtains a hyperbolic metric which realizes given combinatorial structure by a circle pattern on the surface.
The hyperbolic metric obtained is unique up to isometries, and this theorem is also called {\it discrete uniformization theorem} \cite[Chapter 4]{Stephenson}.
In this article, we offer yet another way of endowing a hyperbolic metric on a surface in a manner that given weighted graph is realized as the image of a harmonic map with the least energy among a fixed homotopy class as well as all the hyperbolic metrics.

Let $X=(V, E, m_E)$ be a finite weighted graph with $m_E: E \to (0, \infty)$.
We identify each edge $e$ with the unit interval $[0, 1]$, and denote by $\wbar e$ its reversed edge.
We understand $E$ the set of all oriented edges $e$, $\wbar e \in E$, and the weight function $m_E$ is symmetric, i.e., $m_E(e)=m_E(\wbar e)$.

For any closed Riemann surface $(S, G)$ equipped with a hyperbolic metric $G$,
let $f:X \to (S, G)$ be a piecewise smooth map, i.e., the restriction $f_e:[0, 1] \to (S, G)$ is piecewise smooth for each edge $e$.
We define the {\it energy} of $f$ by
\begin{equation}\label{Eq:Dirichlet}
E_G(f):=\frac{1}{2}\sum_{e \in E}m_E(e)\int_0^1\left\|\frac{d f_e}{dt}(t)\right\|_G^2\,dt.
\end{equation}
It is known that the energy $E_G(f)$ attains the minimum by a harmonic map $h_G$ among the set of all piecewise smooth maps homotopic to $f$ 
(Section \ref{Sec:harmonic}).
A \textit{harmonic map} $h$ from $X$ into a hyperbolic surface $(S, G)$ is a map such that $h_e$ is a (constant speed) geodesic for each $e \in E$, and satisfies the \textit{balanced condition}
\[
\sum_{e \in E_x}m_E(e)\frac{d h_e}{dt}(0)=0, \quad \text{for each $x \in V$},
\]
where $E_x=\{e \in E \ : \ e_0=x\}$, the set of edges emanating from $x$.
Moreover, if a harmonic map $h_G$ is not homotopic to a point nor to a closed curve in $(S, G)$, then $h_G$ is unique as a map due to the negative curvature of $G$.
Then, regarding the energy $E_G(h_G)$ of the harmonic map $h_G$ as a function of hyperbolic metrics $G$ on $S$,
we may further minimize $E_G(h_G)$ given a fixed homotopy class $\Cc$ of $f$.
A central question we consider is whether the energy $E_G(f)$ attains the minimum for a pair $(h_0, G_0)$ of a map $h_0: X \to S$ in $\Cc$ and a hyperbolic metric $G_0$, or not.
We show that under some necessary topological condition on $f:X \to S$, the joint minimum of $E_G(f)$ does exist;
namely, we solve the double minimization problem where the first minimization is the classical calculus of variation and the second minimization is the variation among the hyperbolic metrics.
Furthermore, we identify the hyperbolic metric $G$ when the underlying weighted graph $X$ admits an automorphism group which is compatible with the mapping class group of $S$.

Note that for $i=1, 2$, any two pairs $(h_i, G_i)$ of a hyperbolic metric $G_i$ on $S$ and a harmonic map $h_i:X \to (S, G_i)$ in $\Cc$ have the same energy $E_{G_1}(h_{1})=E_{G_2}(h_{2})$ if there exists an isometry $\f: (S, G_1) \to (S, G_2)$ homotopic to the identity map on $S$ and $h_2=\f \circ h_1$.
It holds that if a pair $(h, G)$ attains the minimum of the energy $E_G(h)$, then the hyperbolic metric $G$ on $S$ is unique up to isometries homotopic to the identity map on $S$ under some condition on the fixed homotopy class $\Cc$ as we see below.

We say that a continuous map $f:X \to S$ {\it fills} $S$ if it is injective and the complement of the image $f$ is a disjoint union of (topological) disks.
For example, any skeleton of triangulation of $S$ gives rise to such $X$ and $f$.
Let us denote the set of all piecewise smooth maps homotopic to $f$ by $\Cc=[f]$, and the space of hyperbolic metrics on $S$ by $\Mc_{-1}(S)$.
We define
\[
\Ec: \Cc \times \Mc_{-1}(S) \to [0, \infty), \qquad (f, G) \mapsto E_G(f).
\]
Note that for any piecewise smooth map $f: X \to (S, G)$, we have $E_G(f)<\infty$, and any continuous map $f$ is homotopic to a piecewise smooth map.

\begin{theorem}\label{Thm:main}
Let $X=(V, E, m_E)$ be a connected finite weighted graph, and $S$ be a closed Riemann surface of genus greater than one.
If $f:X \to S$ is a continuous map which fills $S$,
then there exists a pair $(h_0, G_0)$ which attains the minimum of $\Ec$ on $\Cc \times \Mc_{-1}(S)$,
where $\Cc$ is the homotopy class of $f$. 
Moreover, the hyperbolic metric $G_0$ on $S$ is unique up to isometries homotopic to the identity map on $S$.
\end{theorem}

In fact, we prove a more general theorem when $f$ induces a surjective homomorphism $f_\ast: \pi_1(X, x_0) \to \pi_1(S, f(x_0))$ in Theorem \ref{Thm:general_main}.
The unique hyperbolic metric $G_0$ in Theorem \ref{Thm:main} realizes the least energy harmonic embedding of $X$ into the hyperbolic surface $(S, G_0)$ determined by the weighted graph $X$ and the fixed homotopy class of maps $f: X \to S$.

The problem of minimizing energy \eqref{Eq:Dirichlet} was studied by Colin de Verdi\`ere \cite{ColindeVerdiereComment} for embedding of finite weighted graphs into surfaces.
The problem of minimizing energy also in hyperbolic metrics was suggested by Kotani and Sunada \cite[p.7, Section 2]{KotaniSunadaStandard} and has been reiterated by Sunada \cite[p.124, Section 7.7]{SunadaTopological}. 
They showed the corresponding result in the case of flat tori (of dimension at least $2$);
for a connected finite weighted graph $X$, and for an $n$-dimensional torus $\T^n$, 
if $f:X \to \T^n$ induces a surjective homomorphism from $\pi_1(X)$ to $\pi_1(\T^n)$,
then there exists a pair of flat metric $G_0$ on $\T^n$ and a harmonic map $h_0: X \to (\T^n, G_0)$ homotopic to $f$ such that
the corresponding energy functional $\Ec$ attains the minimum on $\Cc \times \Mc_0(\T^n)$, where $\Cc=[f]$ and $\Mc_0(\T^n)$ is the space of flat metrics on $\T^n$ whose volume is normalized to $1$ (and such a pair is unique up to translations on $\T^n$).
In fact, they proved the result by constructing a flat metric explicitly by using the $1$-homology group of graph $X$, and called the resulting harmonic map $h_0$ into $(\T^n, G_0)$ (or, its equivariant lift on an abelian covering of $X$ into $\R^n$) the \textit{standard realization} of $X$.
Then, they proposed the problem in the case of closed hyperbolic surfaces whether such an energy functional has a minimum as a non-Euclidean analogue to their result in dimension $2$.
It seems that a direct construction of such a pair of metric and harmonic map into closed hyperbolic surfaces would not work by a non-linear nature of target spaces.

Our approach to show the existence of a minimum pair of energy functional $\Ec$ for surfaces is that we define a function $\Ec_\Cc$ associated to $\Ec$ on the Teichm\"uller space $\Tc(S)$.
Namely, since the energy $E_G(f)$ attains the minimum at a harmonic map $h_G$ in the homotopy class $\Cc$ of $f$,
the function $G \mapsto E_G(h_G)$ is defined on the space of hyperbolic metrics $\Mc_{-1}(S)$.
This function naturally yields a function on the Teichm\"uller space $\Tc(S)$ of $S$.
Recalling that the Teichm\"uller space $\Tc(S)$ consists of isotopy classes of pairs $((\SS, G_{\SS}), \f)$ of a hyperbolic surface $(\SS, G_{\SS})$ and a diffeomorphism (a marking) $\f:S \to \SS$, we write $G=\f^\ast G_\SS$ the pull-back metric of $G_\SS$ on $S$.
Then, the function  
\[
\Ec_\Cc([G]):=E_{G}(h_G),
\]
is well-defined on $\Tc(S)$ (Section \ref{Sec:energy}).
We show that the function $\Ec_\Cc$ on $\Tc(S)$ is strictly convex with respect to the Weil-Petersson metric (Theorem \ref{Thm:convex}) and proper if $f$ induces a surjective homomorphism from $\pi_1(X)$ to $\pi_1(S)$ (Theorem \ref{Thm:proper}).
This shows that a minimum of the function $\Ec_\Cc$ on $\Tc(S)$ exists uniquely and we show that it also yields a minimum pair of the original energy functional $\Ec$.

Actually, the result we show is a discrete analogue of a result by Yamada;
he proved that if the domain is a general closed Riemannian manifold $M$ and $f:M \to (S, G)$ is a continuous map to a closed hyperbolic surface inducing a surjective homomorphism between fundamental groups, then the Dirichlet energy functional evaluated at a (unique) harmonic map homotopic to $f$ is proper and strictly convex on the Teichm\"uller space with respect to the Weil-Petersson metric \cite[Theorem 3.2.1]{Yamada} (see also expositions \cite{Yamada2} and \cite{YamadaSugaku}, and the recent paper \cite{Kim_et_al}).
We mostly follow the strategy by Yamada to show the convexity and the properness of energy functional $\Ec_\Cc$; but we give an independent proof for the convexity that enables us to establish an explicit Hessian formula of $\Ec_\Cc$ (Theorem \ref{Thm:Hessian}).
It is also known that such a functional on the Teichm\"uller space $\Tc(S)$ has various applications to, for example, the Nielsen realization problem for the mapping class groups and 
the fact that $\Tc(S)$ is Stein (\cite{KerckhoffNielsen}, \cite{TrombaBook} and \cite{WolpertNielsen}).
Our construction provides a family of proper convex energy functionals on $\Tc(S)$ from finite weighted graphs; 
this also solves the Nielsen realization problem (Remark \ref{Rem:Nielsen}). 
Moreover, since the functional is defined in terms of a finite graph, 
in many cases,
we are able to find explicit hyperbolic surfaces as minima of the energy functionals $\Ec_\Cc$ as fixed points of finite subgroups in the mapping class groups.

More precisely,
let $\Aut(X)$ be the group of automorphisms $\s$ of a finite weighted graph $X=(V, E, m_E)$,
where $\s$ is a bijection to $V$ to itself, preserves edges with $\wbar{\s e}=\s \wbar e$ and $m_E(\s e)=m_E(e)$ for $e \in E$.
Recall that the mapping class group $\Mod(S)$ is the group of isotopy classes of orientation-preserving homeomorphisms.
For any continuous map $f:X \to S$, we define a subgroup of $\Mod(S)$ by
\[
\Gc_{[f]}(S):=\Big\{[\f] \in \Mod(S) \ : \ \text{$\f \circ f \simeq f \circ \s$ for some $\s \in \Aut(X)$}\Big\},
\]
where $f_1 \simeq f_2$ means that $f_1$ and $f_2$ are homotopic.
Note that $\Gc_{[f]}(S)$ is actually a well-defined group which depends only on the homotopy class of $f$.
The group $\Gc_{[f]}(S)$ is possibly trivial, but if it is not trivial and is large enough, then it is used to determine the minimizer of $\Ec_\Cc$ for $\Cc=[f]$.
We show that the group $\Gc_{[f]}(S)$ is finite if $f$ fills $S$ (Lemmas \ref{Lem:fill} and \ref{Lem:finiteness}).

\begin{theorem}\label{Thm:Nielsen}
Let $S$ be a closed Riemann surface of genus greater than one,
$f: X \to S$ be a continuous map which fills $S$ and $\Cc=[f]$ be the homotopy class of $f$.
If $[(\SS_0, G_0), \f_0]$ is the unique minimizer of $\Ec_\Cc:\Tc(S) \to [0, \infty)$,
then
there exists a group $\wt \Gc_{0}$ of isometries on $(S, \f_0^\ast G_0)$ such that
the natural map
\[
\wt \Gc_{0} \to \Gc_{[f]}(S), \quad \wt \f \mapsto [\wt \f],
\]
gives an isomorphism of groups.
Moreover, 
if $h_0: X \to (S, \f_0^\ast G_0)$ is the unique harmonic map in the homotopy class $\Cc=[f]$,
then for any $\wt \f$ in $\wt \Gc_{0}$, there exists $\s_{[\wt \f]}$ in $\Aut(X)$ such that
$\wt \f \circ h_0=h_0 \circ \s_{[\wt \f]}$.
\end{theorem}

The proof is given in Section \ref{Sec:MCG}.
Theorem \ref{Thm:Nielsen} shows that several classical examples of hyperbolic surfaces arise as minima of some energy functionals $\Ec_\Cc$ defined for some finite weighted graphs $X$.
We discuss examples associated to pairs of pants decompositions and triangle tessellations in Section \ref{Sec:examples}.
Let us give one simple example:
for any integer $g \ge 2$,
consider the regular $4g$-gon $F$ with each inner angle $\pi/(2g)$.
Identifying opposite pairs of sides in the (induced) orientation-reversing way,
we obtain a closed surface of genus $g$ endowed with a hyperbolic metric, and denote it by $(S, G_{\reg})$.
We consider the bouquet graph $X=(V, E, m_E)$ with equal positive weight $m_E \equiv m$, where the underlying graph consists of a single vertex and $2g$ self-loops.
If we have the regular $4g$-gon $F$ and take the center of $F$ as a vertex and $2g$ (constant speed) geodesic segments passing through the center such that each line has two extremes lying on the midpoints in each pair of opposite sides, then we obtain an embedding map $f:X \to (S, G_{\reg})$ (Figure \ref{Fig:octagon}). 
Note that $f$ fills the surface $S$ since the complement of the image $f$ is an (open) $4g$-gon.
In this case, $f:X \to (S, G_{\reg})$ is a harmonic map; the image is a union of $2g$ closed geodesics at a point, and thus the balanced condition is automatically satisfied.
We claim that the hyperbolic surface $(S, G_{\reg})$ realizes the unique minimizer of $\Ec_\Cc$ for $\Cc=[f]$ with the harmonic map $f:X \to (S, G_\reg)$.

\begin{figure}[h!]
\centering
\includegraphics[width=80mm]{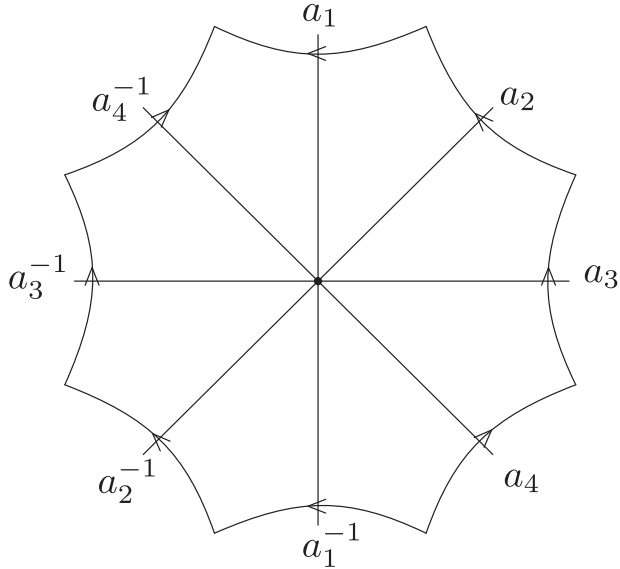}
\caption{A closed surface of genus $2$ with an embedded bouquet graph is obtained if we identify opposite sides of the octagon.}
\label{Fig:octagon}
\end{figure}

Indeed, for the regular $4g$-gon $F$ in the hyperbolic plane $\H^2$, the counter-clock wise rotation about the center of $F$ induces an element $\f$ of the mapping class group of the surface $S$ of order $4g$.
Theorem \ref{Thm:Nielsen} implies that if a surface with a harmonic map $h_0$ homotopic to $f$ realizes the minimizer of $\Ec_\Cc$ for $\Cc=[f]$,
then $\f$ is realized as an isometry and preserves the image of $h_0$.
Then, the hyperbolic surface has a fundamental domain which is a regular $4g$-gon in $\H^2$
and thus it is $(S, G_\reg)$.
Moreover, since the harmonic map $h_0$ is homotopic to $f$, the uniqueness of harmonic map implies that $h_0=f$. 
Therefore $(S, G_\reg)$ gives the unique minimizer of $\Ec_\Cc$ for $\Cc=[f]$ with the harmonic map $f$ as claimed.

%We discuss other examples associated to pairs of pants decompositions and triangle tessellations in Section \ref{Sec:examples}.
Let us point out that Wolpert has studied a sum of length functionals for a filling family of closed curves in \cite{WolpertNielsen}.
There the functional is strictly convex with respect to the Weil-Petersson metric and proper on the Teichm\"uiller space,
and thus has a unique minimizer as our functional does. 
It is a natural question to describe the minimizer of Wolpert's functional.
In our simple example above,
in order to determine the minimizing hyperbolic surface of energy functional, 
we have used the uniqueness of harmonic map in a fixed homotopy class --- 
the fact in general not available for a sum of length functionals --- our approach based on Theorem \ref{Thm:Nielsen} does not apply to the question.
We discuss other explicit examples associated to pairs of pants decompositions and triangle tessellations in Section \ref{Sec:examples}.
%Finally, we conclude this introduction with the following natural question which is of some interest:
%\begin{question}
%Does every point in the Teichm\"muller space $\Tc(S)$ arise as a minimum of some energy functional $\Ec_\Cc$ for some $\Cc=[f]$ and for some finite weighted graph?
%\end{question}

The organization of this paper is the following:
in Section \ref{Sec:energyfunctionalTc}, we introduce and study the main energy functional $\Ec_\Cc$ on the Teichm\"uller space;
first we review results that we use on harmonic maps from finite weighted graphs into closed surfaces, following Kotani and Sunada \cite{KotaniSunadaStandard},
and then state results on the convexity (Theorem \ref{Thm:convex}) and the properness (Theorem \ref{Thm:proper}) of $\Ec_\Cc$,
and show the main result (a generalized form Theorem \ref{Thm:general_main} and Theorem \ref{Thm:main}) by combining these two theorems.
In Section \ref{Sec:convex}, we give the proof of Theorem \ref{Thm:convex} and also give a formula for the Hessian of $\Ec_\Cc$ (Theorem \ref{Thm:Hessian}), and in Section \ref{Sec:proper}, we prove Theorem \ref{Thm:proper}.
In Section \ref{Sec:MCG}, we discuss the action of a finite subgroup of the mapping class group on the Teichm\"uller space and prove Theorem \ref{Thm:Nielsen}.
In Section \ref{Sec:examples}, we provide examples which are minima of some energy functionals $\Ec_\Cc$ by using Theorem \ref{Thm:Nielsen}.
In Appendix, we prove technical results: the first and second variation formulas for finite weighted graphs and the smooth dependency of harmonic maps along a smooth one parameter family of metrics in a general closed Riemannian manifold target.

\section{The energy functional on the Teichm\"uller space}\label{Sec:energyfunctionalTc}

\subsection{Discrete harmonic maps}\label{Sec:harmonic}

Let $S$ be a closed Riemann surface of genus greater than one endowed with a hyperbolic metric $G$, and $X=(V, E, m_E)$ be a finite weighted graph.
Theorems by Colin de Verdi\`ere \cite{ColindeVerdiereComment} and by Kotani-Sunada in \cite{KotaniSunadaStandard} imply that for any continuous map $f:X \to S$, the minimum of energy $E_G$ is achieved by a harmonic map $h$ among all piecewise smooth maps $\Cc$ homotopic to $f$.
Moreover, $h$ is a harmonic map in $\Cc$ if and only if $h$ attains the minimum of the energy $E_G$ in $\Cc$.
We give a proof adapted to our setting for the sake of convenience.

\begin{theorem}[Th\'eor\`eme 1 in \cite{ColindeVerdiereComment} and Theorem 2.2, 2.3 and 2.5 in \cite{KotaniSunadaStandard}]\label{Thm:KSharmonic}
Fix a hyperbolic metric $G$ in $S$ and a  (not necessarily connected) finite weighted graph $X=(V, E, m_E)$.
For any continuous map $f:X \to (S, G)$, let $\Cc=[f]$ be the set of all piecewise smooth maps homotopic to $f$.
\begin{itemize}
\item[(i)] There exists a map $h$ such that the energy $E_G$ is the minimum in $\Cc$.
Moreover, if a map $h$ attains the minimum of $E_G$ in $\Cc$, then $h$ is a harmonic map.
\item[(ii)] For any harmonic maps $h_0$ and $h_1$ in $\Cc$, we have $E_G(h_0)=E_G(h_1)$.
\item[(iii)] If the image of $f$ on each connected component of $X$ is not homotopic to a point nor a closed circle, then there exists a unique harmonic map $h$ in $\Cc$ (as a map).
Moreover, the Hessian ${\rm Hess}_{E_G}$ of the energy functional $E_G$ at $h$ is non-degenerate.
\end{itemize}
\end{theorem}
 
 \def\geo{{\rm geo}}
\proof
Let $\Cc_\geo$ be the set of piecewise geodesic maps in $\Cc$, i.e., maps restricted to each edge $e$ is a geodesic.
Note that $\Cc_\geo$ is non-empty because we find a map $f_\geo \in \Cc_\geo$ for any piecewise smooth map $f\in \Cc$ so that $f(x)=f_\geo(x)$ for any $x\in V$, and furthermore $E_G(f_\geo)\leq E_G(f)$.
If we endow the $C^1$-topology on $\Cc_\geo$, then the energy functional $E_G$ restricted to $\Cc_\geo$ is continuous and proper (i.e., for each $R \ge 0$, the set $\{f \in \Cc_\geo \ : \ E_G(f) \le R\}$ is compact) on $\Cc_\geo$. 
Thus, there exists a piecewise geodesic map $h$ which attains the minimum of $E_G$ in $\Cc$.

Note that a map $h:X \to (S, G)$ is a harmonic map if and only if for any smooth variations $f_s$ for $s\in (-\e, \e)$, $\e>0$ and $f_0=h$, we have $dE_G(f_s)/ds|_{s=0}=0$ by 
the first variation formula \eqref{Eq:firstVF} in Lemma \ref{Lem:first and second} since $h$ satisfies that
\[
\nabla_{\p_t h_{e}}\p_t h_{e}=0 \quad \text{for any $e \in E$} \quad \text{and} \quad \sum_{e \in E_x}m_E(e)\p_t h_{e}(0)=0 \quad \text{for any $x \in V$},
\]
where we write $\p_t h_{e}(t):=(d h_e/dt)(t)$ for $e \in E$ and $t \in [0, 1]$.
Hence if $h$ attains the minimum of $E_G$ in $\Cc$, then such $h$ is always a harmonic map.
This proves (i).

For any two harmonic maps $h_0$ and $h_1$ in $\Cc$,  
there exists a homotopy $f_s: X \to S$ so that $f_s(x)$ for $s \in [0, 1]$
is a (unique) geodesic from $h_0(x)$ to $h_1(x)$ in the hyperbolic surface $(S, G)$ for any $x\in X$ by perturbing the given homotopy if necessary.
We denote the variational vector field of $f_s$ by $V_e^s(t):=\frac{d}{ds} f_{s, e}(t)$ and $T_e^s:=\p_t f_{s, e}(t)$. 
Then, $\nabla_{V_e^s}V_e^s=0$ for the Levi-Civita connection $\nabla$, and hence, the second variation formula \eqref{Eq:secondVF} in Lemma \ref{Lem:first and second} shows that
\begin{align}\label{positive}
\frac{d^2}{ds^2}E_G(f_s)=\sum_{e\in E} m_E(e)\int_0^1\left\{\|\nabla_{T_e^s} V_e^s\|_G^2+\|(V_e^s)^\perp\|_G^2\|T_e^s\|_G^2\right\}\,dt\geq 0
\end{align}
for any $s\in [0,1]$ since $(S,G)$ has constant sectional curvature $-1$, where $(V_e^s)^\perp$ means the orthogonal projection of $V_e^s$ normal to $T_e^s$ if $T_s^s \neq 0$, and $0$ otherwise.
Since $h_0$ and $h_1$ are harmonic maps, we have $dE_G(f_s)/ds|_{s=0}=dE_G(f_s)/ds|_{s=1}=0$, and \eqref{positive} implies that $dE_G(f_s)/ds\equiv 0$ on $[0,1]$ and $E_G(h_0)=E_G(h_1)$.
This shows (ii).

Furthermore, in fact, we have $d^2E_G(f_s)/ds^2\equiv 0$ on $[0, 1]$, and by \eqref{positive}, $\nabla_{T_e^s}V_e^s=0$ and since $(S, G)$ has negative sectional curvature,
$V_e^s$ and $T_e^s$ are proportional for any $s \in [0, 1]$.
If the image of $f$ on each connected component $X^0$ of $X$ is not homotopic to a point nor a closed circle, then for any $s \in [0, 1]$ and in each connected component $X^0=(V^0, E^0)$, there exists a vertex $x \in V^0$ for which $\{T_e^s\}_{e \in E^0_x}$ contains at least two linearly independent vectors in $T_{f_s(x)}S$ since $f_s$ is homotopic to $f$; and thus $V_e^s\equiv 0$ for all $e \in E^0_x$.
In fact, $V_e^s \equiv 0$ for all $e \in E$ by parallel transports along the image of edges since $\nabla_{T_e^s}V_e^s=0$ and each $X^0$ is connected, and this holds for any $s \in [0, 1]$.
Therefore $h_0=h_1$.
Moreover, ${\rm Hess}_{E_G}(V, V)=(d^2/ds^2)|_{s=0}E_G(f_s)=0$ implies $V_e \equiv 0$ for all $e \in E$. 
We conclude the proof of (iii).
\qed

%\begin{remark}
%Note that Theorem \ref{Thm:KSharmonic} (i) and (ii) show that 
%$h$ is a harmonic map in $\Cc$ if and only if $h$ attains the minimum of the energy $E_G$ in $\Cc$. 
%Kotani and Sunada proved that Theorem \ref{Thm:KSharmonic} (i) holds for maps from $X$ to $(M, G)$ for a general compact Riemannian manifold, Theorem \ref{Thm:KSharmonic} (ii) holds if $(M, G)$ has non-positive sectional curvature, and Theorem \ref{Thm:KSharmonic} (iii) holds if $(M, G)$ has negative sectional curvature \cite{KotaniSunadaStandard}.
%\end{remark}

The following proposition guarantees that for a given closed hyperbolic surface $(S, G)$ and a harmonic map $h:X \to (S, G)$, 
harmonic maps $h_s$ depends smoothly on a smooth change of hyperbolic metrics $G_s$.
We give the proof in a more general setting in Proposition \ref{smooth dependence} in Appendix.

\begin{proposition}\label{Prop:smooth}
Let $X$ be a connected finite weighted graph, $(S, G)$ be a closed hyperbolic surface and $h: X \to (S, G)$ be a harmonic map.  Suppose that the image of $h$ is not a point nor a closed geodesic.
Then, for any $\e>0$ and for any smooth family of hyperbolic metrics $\{G_s\}_{s\in (-\epsilon,\epsilon)}$ with $G_0=G$ in $\Mc_{-1}(S)$, there exists a unique family of maps $h_s:X \to S$ for $s\in (-\epsilon,\epsilon)$ with $h_0=h$ such that $h_s:X \to (S, G_s)$ is a harmonic map for every $s \in (-\e, \e)$ and $h_s$ is smooth with respect to the variable $s$.
\end{proposition}

\proof
If the image of $h$ is not homotopic to a point nor a closed geodesic, then Theorem \ref{Thm:KSharmonic} (iii) implies that for each $s \in (-\e, \e)$ there exists a unique harmonic map $h_s: X \to (S, G_s)$ homotopic to $h$ and the Hessian of $E_G$ at $h_s$ is non-degenerate.
Hence by Proposition \ref{smooth dependence}, since $G_s \in \Mc_{-1}(S)$ for all $s \in (-\e, \e)$, for each $G_s$ and for the harmonic map $h_s: X \to (S, G_s)$, there exists an open set $\Uc$ around $G_s$ in $\Mc_{-1}(G)$ and a smooth map $\wt h: \Uc \to C^\infty(X, S)$ (the space of piecewise smooth map from $X$ into $S$) such that $\wt h(G'):X \to (S, G')$ is a harmonic map for all $G' \in \Uc$ and $\wt h(G_s)=h_s$.
Since for each $s \in (-\e, \e)$, the harmonic map $h_s:X \to (S, G_s)$ is unique in the homotopy class, covering the curve $\{G_s\}_{s \in (-\e, \e)}$ by such open sets, we obtain the claim.
\qed

\subsection{The energy functional on the Teichm\"uller space}\label{Sec:energy}
Recall that the Teichm\"uller space of $S$ consists of equivalence classes of pairs $((\SS, G_{\SS}), \f)$ of a hyperbolic surface $(\SS, G_{\SS})$ and a diffeomorphism (a marking) $\f:S \to \SS$,
where
$((\SS_1, G_1), \f_1) \sim ((\SS_2, G_2), \f_2)$ if and only if $\f_2 \circ \f_1^{-1}$ is homotopic to an isometry from $(\SS_1, G_1)$ to $(\SS_2, G_2)$.
We discuss the standard topology in $\Tc(S)$ (see e.g.,  \cite[Section 10.3, p.269]{FarbMargalit}).
The mapping class group $\Mod(S)$ is the group of isotopy classes of orientation-preserving diffeomorphisms on $S$.
The group $\Mod(S)$ acts on the Teichm\"uller space by the change of markings
\[
[\phi]\cdot [(\SS, G_{\SS}), \f]:=[(\SS, G_{\SS}), \f \circ \phi^{-1}] \qquad \text{for $[\phi] \in \Mod(S)$},
\] 
and the action is properly discontinuous \cite[Theorem 12.2, p.350]{FarbMargalit}.
The moduli space $\Mc(S)$ of $S$ is the quotient space of $\Tc(S)$ by the action of the mapping class group $\Mod(S)$.

The energy $E_G(f)$ attains the minimum at a harmonic map $h_G$ in the homotopy class $\Cc$ of $f$.
The function $G \mapsto E_G(h_G)$ on the space of hyperbolic metrics $\Mc_{-1}(S)$ naturally defines a function on the Teichm\"uller space $\Tc(S)$ of $S$.
For a point $[(\SS, G_{\SS}), \f]$ of the Teichm\"uller space $\Tc(S)$,
we write $\f^\ast G_\SS$ the pull-back metric of $G_\SS$ on $S$ by $\f:S \to (\SS, G_\SS)$.
If $h_1: X \to (S, \f_1^\ast G_1)$ is harmonic and $((\SS_1, G_1), \f_1) \sim ((\SS_2, G_2), \f_2)$, 
then the identity map on $S$ is homotopic to an isometry $\i: (S, \f_1^\ast G_1) \to (S, \f_2^\ast G_2)$,
and $\i \circ h_1:X \to (S, \f_2^\ast G_2)$ is also a harmonic map homotopic to $h_1$; hence
for any harmonic map $h_2: X \to (S, \f_2^\ast G_2)$ homotopic to $h_1$, one has
\[
E_{\f_1^\ast G_1}(h_1)=E_{\f_2^\ast G_2}(h_2)=\min_{h \in \Cc}E_G(h),
\]
by Theorem \ref{Thm:KSharmonic} (i) and (ii).
We define
\[
\Ec_\Cc([G]):=E_{G}(h_G),
\]
where $G=\f^\ast G_{\SS}$ is the hyperbolic metric induced on $S$ by a pair $((\SS, G_{\SS}), \f)$ representing a point in $\Tc(S)$ and $h_G$ is a harmonic map from $X$ to $(S, G)$ in the homotopy class $\Cc=[f]$.
Note that $\Ec_\Cc$ is a function on the Teichm\"uller space
\[
\Ec_\Cc:\Tc(S) \to [0, \infty), \quad \text{and} \quad \Ec_\Cc([G])=E_{G}(h_G)=\min_{h \in \Cc}E_G(h).
\]
We show that the function $\Ec_\Cc$ on $\Tc(S)$ is strictly convex with respect to the Weil-Petersson metric;
we will give the proof in Section \ref{Sec:convex}.

\begin{theorem}\label{Thm:convex}
Let $X=(V, E, m_E)$ be a connected finite weighted graph and $S$ be a closed Riemann surface of genus greater than one.
For any homotopy class $\Cc$ of piecewise smooth maps $f: X \to S$ such that the image of $f$ is not homotopic to a point nor a closed geodesic,
the function $\Ec_\Cc([G])=\min_{f \in \Cc}E_G(f)$ on $\Tc(S)$ is strictly convex with respect to the Weil-Petersson metric.
\end{theorem}

Recall that a continuous map $f:X \to S$ fills $S$ if there exists an injective map $f_1:X \to S$ homotopic to $f$ such that the complement of the image $f_1$ is a disjoint union of disks.
\begin{lemma}\label{Lem:fill}
If $f:X \to S$ fills $S$, then the induced homomorphism $f_\ast: \pi_1(X, x_0) \to \pi_1(S, f(x_0))$ is surjective.
\end{lemma}
\proof
There exists an injective map $f_1$ homotopic to $f$ such that the complement of the image $f_1$ is a disjoint union of disks.
Then, for each disk, choosing a point from the interior, we take a homotopy on the disk such that it pushes outside of a small neighborhood of the point to the boundary and remains identity on the boundary. 
Patching these homotopies together, we obtain a homotopy on the surface.

For any point $x_0$ in $X$, and any loop $\g$ based at $f_1(x_0)$ in $S$, up to a small perturbation of $\g$ if necessary, composing the homotopy constructed, we obtain a loop confined in the boundaries of disks.
Since $f_1$ is injective, there is a loop $c$ in the graph $X$ based at $x_0$, whose image by $f_1$ is homotopic to the original loop $\g$ relative to $f_1(x_0)$ on the surface. This shows that the loop $c$ satisfies $f_{1 \ast} [c]=[\g]$. 
Thus, $f_{1 \ast}: \pi_1(X, x_0) \to \pi_1(S, f_1(x_0))$ is surjective, and since $f_1$ is homotopic to $f$, we conclude the claim.
\qed

If we fix the homotopy class $\Cc$ of $f$ such that $f$ induces a surjective homomorphism from $\pi_1(X)$ to $\pi_1(S)$, 
then the energy functional $\Ec_\Cc$ is proper; we shall give the proof in Section \ref{Sec:proper}.

\begin{theorem}\label{Thm:proper}
Let $X=(V, E, m_E)$ be a finite weighted graph, and $S$ be a closed Riemann surface of genus greater than one.
Let $\Cc$ be the homotopy class of a continuous map $f:X \to S$ such that the induced homomorphism $f_\ast:\pi_1(X, x_0) \to \pi_1(S, f(x_0))$ is surjective.
Then, the function $\Ec_\Cc:\Tc(S) \to [0, \infty)$ is proper, i.e., for any $R \ge 0$, the sublevel set of $\Ec_\Cc$,
\[
\Big{\{}[(\SS, G_\SS), \f] \in \Tc(S) \ : \ \Ec_\Cc([G]) \le R, \ G=\f^\ast G_\SS \Big{\}}
\]
is compact relative to the standard topology in $\Tc(S)$.
\end{theorem}

In general, we are not able to remove the condition that $f$ induces a surjective homomorphism from $\pi_1(X)$ to $\pi_1(S)$ in Theorem \ref{Thm:proper}.
For example, if we have a simple closed curve in $S$ as an image of $f$, then
one is able to make the length of closed geodesic homotopic to $f$ arbitrary small by changing hyperbolic metrics in $S$.

\begin{theorem}\label{Thm:general_main}
Let $X=(V, E, m_E)$ be a finite weighted graph, and $S$ be a closed Riemann surface of genus greater than one.
If $\Cc=[f]$ is the homotopy class of $f$ which induces a surjective homomorphism from $\pi_1(X)$ to $\pi_1(S)$,
then the energy functional $\Ec_\Cc$ has a unique minimum point in the Teichm\"uller space $\Tc(S)$.
\end{theorem}
\proof
Theorem \ref{Thm:proper} implies that there exists a minimum of $\Ec_\Cc$ on $\Tc(S)$ if $f$ induces a surjective homomorphism from $\pi_1(X)$ to $\pi_1(S)$. 
Since any two points in $\Tc(S)$ can be connected by a Weil-Peterson geodesic by a theorem by Wolpert \cite[Corollary 5.6]{WolpertNielsen},
Theorem \ref{Thm:convex} shows that a minimum is unique.
\qed

\proof[Proof of Theorem \ref{Thm:main}]
Lemma \ref{Lem:fill} and Theorem \ref{Thm:general_main} implies that there exists a unique point $[(\SS, G_{\SS}), \f]$ in the Teichm\"uller space $\Tc(S)$ such that $\Ec_\Cc$ attains the minimum.
Since $\Ec_\Cc([G])=E_{G}(h_G)=\min_{h \in \Cc}E_G(h)$ for a harmonic map $h_G:X \to (S, G)$ in the homotopy class $\Cc$ and $G=\f^\ast G_{\SS}$,
the pair $(h_G, G)$ of a harmonic map $h_G$ in $\Cc$ and the hyperbolic metric $G$ on $S$ attains the minimum $E_G(h_G)$.
Note that the hyperbolic metric $G \in \Mc_{-1}(S)$ is unique up to isometries homotopic to the identity on $S$, we conclude the claim.
\qed

\section{Convexity of the energy functional}\label{Sec:convex}

In this section, we compute the second derivative of the energy functional $\Ec_\Cc$ for $\Cc=[f]$. 
The argument is inspired by Wolf, who showed the convexity of the {\it length} functional of a {\it closed} curve along the Weil-Petersson geodesic  \cite{Wolf}.
We discuss the energy of curves and extend it to any finite weighted graphs.
First, we recall some differential geometric aspects and facts that we use on the Teichm\"uller space, following \cite{Yamada}, \cite{Yamada2} and \cite{Wolf}.

\subsection{Hessian formula of the energy functional}\label{ss1}

Let ${\rm Met}(S)$ be the set of all smooth metrics on $S$.
Then, $\mathcal{M}_{-1}(S)$ is the subset of ${\rm Met}(S)$ consisting of metrics of constant sectional curvature $K=-1$. 
The tangent spaces of ${\rm Met}(S)$ 
is regarded as
\begin{align*}
T_G{\rm Met}(S)=\Big{\{}\textup{symmetric}\ (0,2)\textup{-tensors on}\ S\Big{\}}.
\end{align*}
The {Teichm\"uller space} $\mathcal{T}(S)$ of $S$ is identified with the quotient space of $\mathcal{M}_{-1}(S)$ by ${\rm Diff}_0(S)$ via the map
\[
\mathcal{T}(S) \to \mathcal{M}_{-1}(S)/{\rm Diff}_0(S), \qquad [(\Sigma, G_\SS), \varphi]\mapsto [\varphi^\ast G_\SS], 
\]
where ${\rm Diff}_0(S)$ is the identity component of the group of diffeomorphisms on $S$.

The deformations arising from ${\rm Diff}_0(S)$ are given by the set of Lie derivatives $L_XG$ of $G$ for smooth vector fields 
$X$ on $S$. 
The tangent space to $\mathcal{T}(S)$ at $[G]$ is identified with
\begin{align*}
T_{[G]}\mathcal{T}(S)&=\{Q\in T_G{\rm Met}(S);\ {\rm tr}_GQ=0\ {\rm and}\ \delta_GQ=0\},
\end{align*}
where $\delta_GQ:=-{\rm tr}_{12}(\nabla Q)=-G^{ij}\nabla{Q}(\p_i,\p_j, \cdot)$ for the Levi-Civita connection $\nabla$ of $(S, G)$ \cite[Section 3.2]{Yamada2}, where for a local coordinate $(x^1,x^2)$ on $S$, we abbreviate $\p/\p x^1$ and $\p/\p x^2$ by $\p_1$ and $\p_2$, respectively.  
We denote the matrix representation of any symmetric $(0,2)$-tensor $Q$ by $Q=(Q_{ij})_{i,j=1,2}=(Q(\p_i,\p_j))_{i, j=1, 2}$.

 \begin{lemma}\label{Lem:Qholomorphic}
 Let $(S,G)$ be a hyperbolic surface of constant sectional curvature $K=-1$ and $Q$ be a symmetric $(0,2)$-tensor on $S$ such that ${\rm tr}_G Q=0$ and $\delta_GQ=0$. 
 Then, we have the following:
 \begin{itemize}
 \item[(i)] For any isothermal chart $(U, z=x^1+\sqrt{-1}x^2)$ of $(S,G)$, the function $Q_{11}-\sqrt{-1}Q_{12}$ is holomorphic on $U$.  
 \item[(ii)] It holds that
\begin{align*}
(\Delta_G+4)\|Q\|^2_G\geq 0,
\end{align*}
where $\Delta_G$ is the Laplace-Beltrami operator acting on $C^{\infty}(S)$.
 \end{itemize}
 \end{lemma}
 \proof
Take an isothermal coordinate $(U, z=x^1+\sqrt{-1}x^2)$ so that the metric $G$ is expressed by $G(z)=\rho(z)|dz|^2=\rho(x^1, x^2)((dx^1)^2+(dx^2)^2)$ for a positive real function $\rho$ on $U$. Note that the Christoffel symbol $\Gamma_{ij}^k$  for the Levi-Civita connection on any isothermal coordinate satisfies
\begin{align*}
\begin{cases}
\Gamma_{11}^1=\Gamma_{12}^2=\Gamma_{21}^2=-\Gamma_{22}^1=\frac{1}{2}\p_1\log {\rho},\\
\Gamma_{22}^2=\Gamma_{21}^1=\Gamma_{12}^1=-\Gamma_{11}^2=\frac{1}{2}\p_2\log{\rho},\\
\end{cases}
\end{align*} 
Then we have
\begin{align*}
(\delta_{G}Q)(\p_k)=-\rho^{-1}(Q_{1k,1}+Q_{2k,2}),
\end{align*}
for $k=1,2$ on $U$, where $Q_{ij,k}:=\p_kQ_{ij}$. Using $\delta_{G}Q=0$, $Q_{22}=-Q_{11}$ and $Q_{12}=Q_{21}$,  
we obtain
\[
Q_{11,1}=-Q_{12,2}\quad {\rm and}\quad Q_{11,2}=Q_{12,1},
\]
hence, $Q_{11}-\sqrt{-1}Q_{12}$ is a holomorphic function on $U$ since it satisfies the Cauchy-Riemann equation.
This shows (i).

Let us show (ii).
If $\|Q\|_G^2\equiv 0$ or equivalently $Q\equiv 0$, then the statement is obvious. We may assume the zeros of the function $\|Q\|_G^2$ is isolated since $Q_{11}-\sqrt{-1}Q_{12}$ is holomorphic by (i).
Take any $p\in S$ with $\|Q\|_G^2(p)\neq 0$ and an isothermal coordinate $(U, z=x^1+\sqrt{-1} x^2)$ around $p$ so that  $G=\rho((dx^1)^2+(dx^2)^2)$ and $\|Q\|_G^2\neq 0$ on $U$.   
Note that,  we have $\Delta_G=\rho^{-1}(\p_1^2+\p_2^2)$ in the isothermal coordinate. 

We have
\[
\|Q\|_G^2=2\rho^{-2}(Q_{11}^2+Q_{12}^2)
\]
since $Q$ is symmetric and trace free tensor. 
Taking the logarithm of this equation, we have $\log\|Q\|_G^2=\log|Q|^2-2\log\rho+\log2$, where we set $|Q|^2:=|Q_{11}-\sqrt{-1}Q_{12}|^2=Q_{11}^2+Q_{12}^2$,  and hence, we obtain
\begin{align}\label{dlog}
\Delta_G\log\|Q\|_G^2&=\Delta_{G}\log|Q|^2-2\Delta_G\log\rho
\end{align} 
Since $Q_{11}-\sqrt{-1}Q_{12}$ is holomorphic on $U$ by (i), we see $\Delta_{G}\log|Q|^2=0$ on $U$. 
On the other hand,  
in the isothermal coordinate, 
the sectional curvature $K$ is given by $K=-(1/2)\Delta_G\log \rho$. Since we assume $K=-1$, the equation \eqref{dlog} becomes
\[
\Delta_G\log\|Q\|_G^2=-4.
\]
By using the formula $\Delta_G\log F=\frac{\Delta_G F}{F}-\frac{ |\nabla F|_G^2}{F^2}$, we obtain
\[
\frac{\Delta_G\|Q\|_G^2}{\|Q\|_G^2}+4=\frac{|\nabla \|Q\|_G^2|_G^2}{\|Q\|_G^4}\geq 0,
\]
as required, and conclude the proof of (ii).
\qed

\begin{remark}\label{Rem:QD}
Given an isothermal coordinate $z=x^1+\sqrt{-1}x^2$ on $(S,G)$ with $G(z)=\rho(z)|dz|^2$, we define $\Phi:=(Q_{11}-\sqrt{-1}Q_{12})dz^2$. It turns out that $\Phi$ defines a global section of $(T^{1,0}S)^*\otimes (T^{1,0}S)^*$ and the section is holomorphic by Lemma \ref{Lem:Qholomorphic} (1). The section $\Phi$ is a {\it holomorphic quadratic differential} on the Riemann surface $S$. 
One verifies  that there exists a one-to-one correspondence between $Q$ and $\Phi$.
\end{remark}

We define an $L^2$-pairing on $T_G{\rm Met}(S)$  by
\begin{align*}
\langle Q_1, Q_2\rangle:=\int_{S} \langle Q_1(x), Q_2(x)\rangle_{G(x)}\,d\mu_G(x) \quad \text{for $Q_1,Q_2\in T_G{\rm Met}(S)$},
\end{align*}
where $\m_G$ is the area form with respect to $G$.
It induces an $L^2$-pairing on $T_{[G]}\mathcal{T}(S)$ and a Riemannian metric on $\mathcal{T}(S)$.
This metric is called the {\it Weil-Petersson metric} on $\Tc(S)$.  
Then, the mapping class group $\Mod(S)$ acts on $\Tc(S)$ as isometries of the Weil-Petersson metric.

Let $[G_{s}]\in \mathcal{T}(S)$ be a geodesic relative to the Weil-Petersson metric for $s \in (-\epsilon, \epsilon)$ and for $\e>0$.
Then, we can lift $\{[G_{s}]\}_{s \in (-\e, \e)}$ to a horizontal geodesic $\{G_{s}\}_{s \in (-\e, \e)}$ in $\mathcal{M}_{-1}(S)$ since $\Mc_{-1}(S)\to \Tc(S)$ is a Riemannian submersion (\cite[Proposition 2.3.1]{Yamada}, and see also \cite[Section 3.6]{Yamada2}).  
By Proposition \ref{Prop:smooth},
there exists a family of harmonic maps $h_{s}: X \to (S, G_s)$ for $s \in (-\e, \e)$ associated with $G_{s}$ in the fixed homotopy class $\mathcal{C}=[f]$
such that $h_{s}$ is differentiable with respect to $s \in (-\e, \e)$.
Let us define the associated variational vector field
\begin{align*}
V_e(s,t)&:=\frac{\p }{\p s}h_{s, e}(t) \quad \text{for each edge $e\in E$},
\end{align*}
where $h_{s, e}:[0, 1] \to S$ is the restriction of $h_s$ on the edge $e$.

Recall that the energy functional $\mathcal{E}_{\mathcal{C}}$ for $[G_{s}]$ is given by
\begin{align*}
\mathcal{E}_{\mathcal{C}}([G_{s}]):=E_{G_{s}}(h_{s})&=
\frac{1}{2}\sum_{e\in E}m_E(e)\int_{0}^1(h_{s}^*G_{s})(\p_{t},\p_{t})dt,
\end{align*}
where and henceforth we use the short hand notation $\p_t=d/dt$ on $[0, 1]$ for the sake of simplicity.

We shall consider the first and second derivatives of $\mathcal{E}_{\mathcal{C}}([G_{s}])$.
First, for any fixed $s \in (-\e, \e)$, we have
\begin{align}\label{harm}
0=\frac{\p}{\p u} E_{G_{s}}(h_{u})\Big{|}_{u=s}=\frac{1}{2}\sum_{e\in E}m_E(e)\int_{0}^1h_{s}^*(L_{V_e(s)}G_{s})(\p_{t},\p_{t})dt.
\end{align}
since $h_{s}: X\to (S, G_{s})$ is harmonic for any $s \in (-\e, \e)$ and the energy $E_{G_{s}}(h_{u})$ is minimum at $u=s$ by Theorem \ref{Thm:KSharmonic} (i) and (ii).
Therefore we have that
\begin{align}\label{Eq:harmQ}
\frac{d}{d s}\mathcal{E}_{\mathcal{C}}([G_{s}])&=\frac{1}{2}\sum_{e\in E}m_E(e)\int_{0}^1{\p_s}(h_{s}^*G_{s})(\p_{t},\p_{t})dt\nonumber\\
&=\frac{1}{2}\sum_{e\in E}m_E(e)\int_{0}^1\Big{\{}h_{s}^*\({\p_s}{G}_{s}\)+h_{s}^*(L_{V_e(s)}G_{s})\Big{\}}(\p_{t},\p_{t})dt. \nonumber\\ 
&=\frac{1}{2}\sum_{e\in E}m_E(e)\int_{0}^1h_{s}^*\({\p_s}{G}_{s}\)(\p_{t},\p_{t})dt,   
\end{align}
where we have used \eqref{harm} in the last equality.
The above computation yields a characterization of critical points for $\Ec_\Cc$.
We prove the following proposition regarding the first variation of $\Ec_\Cc$ although we will not make use of it in this paper.

\begin{proposition}
For any $[G]$ in $\Tc(S)$,
$[G]$ is a critical point of $\Ec_\Cc$ on $\Tc(S)$ if and only if any harmonic map $h:X \to (S, G)$ in the homotopy class $\Cc$ satisfies that
\[
\sum_{e \in E}m_E(e)\int_0^1 \F(h_{\ast}\p_t, h_{ \ast}\p_t)\,dt=0,
\]
for any holomorphic quadratic differentials $\F$ on $(S, G)$.
\end{proposition}

\proof
It follows that $[G] \in \Tc(S)$ is a critical point of $\Ec_\Cc$ if and only if for any harmonic map $h:X \to (S, G)$ in the homotopy class $\Cc$, one has
\[
\sum_{e \in E}m_E(e)\int_0^1 Q(h_{ \ast}\p_t, h_{\ast}\p_t)\,dt=0,
\]
for any symmetric $(0, 2)$-tensors $Q$ such that $\tr_{G}Q=0$ and $\d_{G}Q=0$
from \eqref{harm} and \eqref{Eq:harmQ}.
In any isothermal coordinate $(U, z=x^1+\sqrt{-1}x^2)$ on $(S, G)$, we have
\[
Q(h_{ \ast}\p_t, h_{\ast}\p_t)={\rm Re}\,\left\{(Q_{11}-\sqrt{-1}Q_{12})(h(t))\cdot h'(t)^2\right\},
\]
where we regard $h(t)$ as a complex valued function in the right hand side.
Since there is a one-to-one correspondence between $Q \in T_{[G]}\Tc(S)$ and $\F$ holomorphic quadratic differentials on $(S, G)$, where $\F(z)=\f(z)dz^2$ on $U$ for some holomorphic function $\f$ (Remark \ref{Rem:QD}), and if $\F$ is a holomorphic quadratic differential, then $\sqrt{-1}\F$ is also so;
thus we obtain the claim.
\qed

Differentiating \eqref{Eq:harmQ} in $s$ once more, we obtain
\begin{align}\label{esvf}
\frac{d^2}{d s^2}\mathcal{E}_{\mathcal{C}}([G_{s}])\Big{|}_{s=0}
&=\frac{1}{2}\sum_{e\in E}m_E(e)\int_{0}^1h_0^*\Big{\{}{\p_s^2}G_{s}\Big{|}_{s=0}+L_{V_e(0)}\Big({\p_s}G_{s}\Big{|}_{s=0}\Big)\Big{\}}(\p_{t},\p_{t})dt.
\end{align}
Here, the second term of the right hand side of \eqref{esvf} is arranged as follows due to the harmonicity:

\begin{lemma}
\begin{align}
\label{harm3}
\frac{1}{2}\sum_{e\in E}m_E(e)\int_{0}^1 h_0^*\Big(L_{V_e}{\p_s}G_{s}\Big{|}_{s=0}\Big)(\p_{t},\p_{t})dt=-\frac{d^2}{d s^2}E_{G_{0}}(h_{s})\Big{|}_{s=0}.
\end{align}
\end{lemma}

\begin{proof}
Differentiating \eqref{harm} with respect to the variable $s$ at $s=0$,
we have
\begin{align}\label{e31}
0&=\frac{1}{2}\sum_{e\in E}m_E(e)\int_{0}^1h^*_0\Big{\{}{\p_s}(L_{V_e(s)}G_{s})\Big{|}_{s=0}
+(L_{V_e(s)}L_{V_e(s)}G_{s}){|}_{s=0}\Big{\}}(\p_{t},\p_{t})dt\nonumber\\
&=\frac{1}{2}\sum_{e\in E}m_E(e)\int_{0}^1h_0^*\Big{\{}\Big(L_{V_e}{\p_s}G_{s}\Big{|}_{s=0}\Big)+2(L_{V_e}L_{V_e}G_{s})|_{s=0}\Big{\}}(\p_{t},\p_{t})dt,
\end{align}
where we used the following identity in the second equality:
\begin{align*}
h_s^*({\p_s}L_{V_e(s)}G_{s})
&=\frac{\p }{\p s}\frac{\p}{\p u}h_{u+s}^*G_{s}\Big{|}_{u=0}\\
&=\frac{\p }{\p u}\frac{\p}{\p s}h_{u+s}^*G_{s}\Big{|}_{u=0}\nonumber\\
&=\frac{\p}{\p u}\Big(h_{u+s}^*\frac{\p }{\p s}G_{s}+h^*_{u+s}L_{V_e}G_{s}\Big)\Big{|}_{u=0}\nonumber\\
&=h_{s}^*\Big(L_{V_e}\frac{\p }{\p s}G_{s}\Big)+h_{s}^*L_{V_e}L_{V_e}G_{s}.\nonumber
\end{align*}

On the other hand, a similar computation shows that
\begin{align}\label{e32}
\frac{d^2}{d s^2}E_{G_{0}}(h_{s})\Big{|}_{s=0}
&=\frac{1}{2}\sum_{e\in E} m_E(e) \int_0^1 h_0^*\Big\{(L_{V_e}L_{V_e}G_{0})|_{s=0}+\p_s(L_{V_e(s)}G_{0})|_{s=0}\Big\}(\p_t,\p_t)\,dt\nonumber\\
&=\sum_{e\in E} m_E(e) \int_0^1 h_0^*(L_{V_e}L_{V_e}G_{0})|_{s=0}(\p_t,\p_t)\,dt,
\end{align}
where we have used $h_0^\ast \p_s(L_{V_e(s)}G_{0})|_{s=0}=h_{0}^*L_{V_e}L_{V_e}G_{0}$ for $h_{0}^*\Big(L_{V_e}\frac{\p }{\p s}G_{0}\Big)=0$ in the last equality.
Combining \eqref{e31} with \eqref{e32}, we obtain \eqref{harm3}.
\end{proof}

Substituting \eqref{harm3} to \eqref{esvf}, we have that 
\begin{align}\label{esvf2}
\frac{d^2}{d s^2}\mathcal{E}_{\mathcal{C}}([G_{s}])\Big{|}_{s=0}
&=\frac{1}{2}\sum_{e\in E}m_E(e)\int_{0}^1h_0^*\(\frac{d^2}{d s^2}G_{s}\)\Big{|}_{s=0}(\p_t,\p_t)dt
-\frac{d^2}{d s^2}E_{G_{0}}(h_{s})\Big{|}_{s=0}.
\end{align}
The first term in the right hand side of \eqref{esvf2} can be computed by using the differential geometry of Teichm\"uller space (without any assumption for $h_{s}$).

\begin{proposition}[Theorem 3.8 in \cite{Yamada2} (cf.\ Theorem 2.3.2 in \cite{Yamada})]\label{key1}
For any horizontal lift $\{G_{s}\}_{s \in (-\e, \e)}$ of
a geodesic $\{[G_{s}]\}_{s \in (-\e, \e)}$ for $\e>0$ in the Teichm\"uller space $\Tc(S)$ relative to the Weil-Petersson metric with velocity vector $Q:=(\p/\p s)G_s|_{s=0}$, we have
\begin{align*}
\frac{d^2}{d s^2}G_{s}\Big{|}_{s=0}&=\(\frac{1}{4}-\frac{1}{2}(\Delta_{G_0}-2)^{-1}\)\|Q\|_{G_0}^2\cdot G_0
+L_{Z}G_0,
\end{align*}
for some (auxiliary) vector filed $Z$ on $S$. 
\end{proposition}

Thus, the rest is to compute the second derivative of ${E}_{G_0}(h_{s})$. 
For each edge $e\in E$, let
\[
T_e^s(t):=\p_th_{s, e}(t),
\]
and $T_e(t)=T_e^0(t)$ when $s=0$ for simplicity.

\begin{definition}\label{Def:V}
For any velocity vector $Q:=(\p_s G_s)|_{s=0}$, we define a $1$-form on $S$ along $h_{0, e}(t)$ by
\[
Q_e:=Q(T_e, \cdot \, ) \quad \text{for any $e \in E$},
\]
and we denote its $G_0$-metric dual by $Q_e^{\sharp}$, that is, $Q_e^{\sharp}$ satisfies that $G_0(Q_e^{\sharp}, \cdot)=Q_e$. Furthermore, we define a vector field along $h_{0, e}(t)$ by
\begin{align}\label{VW}
\mathcal{V}_e:=
\frac{1}{\|T_e\|_{G_0}}\Big(\nabla_{T_e}V_e+\frac{Q_e^{\sharp}}{2}\Big)\quad \textup{if}\ \|T_e\|_{G_0}\neq 0,
\end{align}
and
$\mathcal{V}_e\equiv 0\ \textup{if}\ \|T_e\|_{G_0}=0$
for any $e \in E$.
\end{definition}

The following proposition holds for any horizontal lift of a smooth curve in $\Tc(S)$ (not necessarily a Weil-Petersson geodesic).

\begin{proposition}\label{key2}
Let $\{[G_s]\}_{s \in (-\e, \e)}$ for $\e>0$ be any smooth curve in $\Tc(S)$.
For any horizontal lift $\{G_s\}_{s \in (-\e, \e)}$ of $\{[G_s]\}_{s \in (-\e, \e)}$ with velocity vector $Q:=(\p/\p s)G_s|_{s=0}$, and for any smooth family of harmonic maps $h_{s}:X \to (S, G_s)$ for $s \in (-\e, \e)$ with the variational vector field $V$, we have
\begin{align*}
\frac{d^2}{d s^2}E_{G_0}(h_{s})\Big{|}_{s=0}&=\frac{1}{2}\sum_{e\in E}m_E(e)\int_0^1G_0(\nabla_{T_e}Q_e^{\sharp}, V_e)\, dt\nonumber\\
&=\frac{1}{2}\sum_{e\in E}m_E(e)\int_0^1\Big(
\frac{1}{4}\|Q\|_{G_0}^2-2\|V_e^{\perp}\|_{G_0}^2-2\|\mathcal{V}_e\|_{G_0}^2
\Big)\cdot \|T_e\|_{G_0}^2\,dt,\nonumber
\end{align*}
where $V_e^{\perp}$ denotes the $G_0$-normal component of $V_e$ with respect to $T_e$ if $T_e \neq 0$ and $0$ otherwise, and $\mathcal{V}_e$ is the vector field on $(S, G_0)$ defined in Definition \ref{Def:V}.
\end{proposition}

\begin{remark}
By the formula given in Lemma \ref{le3} in the next subsection, the variational vector field $V_e(t)$ along $h_{0, e}(t)$ is determined by the velocity vector $Q=(\p_sG_s)|_{s=0}$ and the second order ordinary differential equation 
\begin{align}\label{deqv1}
\nabla_{T_e}\nabla_{T_e}V_e+R(V_e, T_e)T_e=-\frac{1}{2}\nabla_{T_e}Q_e^{\sharp}
\end{align}
with initial vector $V_e(x)$ and $\nabla_{T_e}V_e(x)$ for $x=h_{0, e}(0)$, where $R$ is the curvature tensor of $(S,G_0)$. 
Note that the variational vector filed $V$ is determined uniquely if we take a Weil-Petersson geodesic $\{[G_s]\}_{s \in (-\e, \e)}$ and an initial harmonic map $h_0: X\to S$ whose image is not a point nor closed circle by the uniqueness of the harmonic map.
Moreover, the vector $\mathcal{V}_e(t)$ defined by \eqref{VW} is a solution of the ordinary differential equation
\begin{align}\label{deqv2}
\nabla_{T_e}\mathcal{V}_e=-\frac{1}{\|T_e\|_{G_0}}R(V_e,T_e)T_e.
\end{align}
with initial vector $\mathcal{V}_e(h_{0, e}(0))$ if $\|T_e\|_{G_0}\neq 0$. 
\end{remark}

\def\Hess{{\rm Hess}}

We postpone the proof of Proposition \ref{key2} for the moment. Then, for any Weil-Petersson geodesic $\{[G_s]\}_{s \in (-\e, \e)}$ in $\Tc(S)$, the second variation formula \eqref{esvf2} becomes
\begin{align}\label{svff}
&\frac{d^2}{d s^2}\mathcal{E}_{\mathcal{C}}([G_{s}])\Big{|}_{s=0}\nonumber\\
&=\sum_{e\in E}m_E(e)\int_0^1\Big{\{}-\frac{1}{4}(\Delta_{G_0}-2)^{-1}\|Q\|_{G_0}^2+\|V_e^{\perp}\|_{G_0}^2+\|\mathcal{V}_e\|_{G_0}^2\Big{\}}\cdot \|T_e\|^2\,dt,
\end{align}
by Propositions \ref{key1} and \ref{key2}, and using the fact
\begin{align*}
\frac{1}{2}\sum_{e\in E}m_E(e)\int_0^1 h_0^*L_ZG(T_e,T_e)dt=0
\end{align*}
due to the harmonicity of $h_0$ (see \eqref{harm} for a similar identity).  More generally, we obtain the following Hessian formula:

\begin{theorem}\label{Thm:Hessian}
Let $\{[G_{r,s}]\}_{r, s \in (-\e, \e)}$ be any two-parameter family of Weil-Petersson geodesics in $\mathcal{T}(S)$ with $[G_0]=[G_{0,0}]$ with a horizontal lift $\{G_{r,s}\}_{r, s \in (-\e, \e)}$ in $\mathcal{M}_{-1}(S)$ for $\e>0$. 
For the tangent vectors $P:=(\p_r G_{r,0})|_{r=0}$ and  $Q:=(\p_s G_{0,s})|_{s=0}$, the Hessian of the energy functional $\mathcal{E}_{\mathcal{C}}: \mathcal{T}(S)\to \mathbb{R}$ at $[G_0]$ is given by
\begin{align*}
&\Hess_{\Ec_\Cc}(P,Q)_{[G_0]}=\frac{\p^2}{\p r\p s}\mathcal{E}_{\mathcal{C}}([G_{r,s}])\Big{|}_{r=s=0}\\
&=\sum_{e\in E}m_E(e)\int_0^1\Big{\{}-\frac{1}{4}(\Delta_{G_0}-2)^{-1}\langle P, Q\rangle_{G_0}
+\langle V_e^{\perp}, W_e^{\perp}\rangle_{G_0}+\langle \mathcal{V}_e, \mathcal{W}_e\rangle_{G_0}\Big{\}}\cdot \|T_e\|^2_{G_0}dt,
\end{align*}
where $h_{0}:X \to (S, G_{0})$ is any smooth harmonic map, $T_e(t):=\p_th_{0, e}(t)$, $V_e(t):=\p_r h_{r, e}(t)$ for the family of harmonic maps $h_r: X\to (S, G_{r,0})$, $\Vc_e(t)$ are defined by \eqref{VW} relative to the variable $r$, and $W_e(t)$, 
$\mathcal{W}_e(t)$ are defined analogously relative to the variable $s$.
In the integral, $\langle P, Q\rangle_{G_0}$ stands for $G_0(P, Q)$.
\end{theorem}

One can check this formula by using \eqref{svff}, and thus, we omit the proof. 
Now, we give a proof of strictly convexity of the energy functional $\mathcal{E}_{\mathcal{C}}$.

\proof[Proof of Theorem \ref{Thm:convex}]
We show that
\begin{align}\label{Lem: Wolf}
-\frac{1}{4}(\Delta_{G_0}-2)^{-1}\|Q\|^2_{G_0}\geq \frac{1}{24}\|Q\|^2_{G_0}\geq 0.
\end{align}
Indeed, setting  $u:=(\Delta_{G_0}-2)^{-1}\|Q\|_{G_0}^2+\frac{1}{6}\|Q\|_{G_0}^2$, we see
\[
(\Delta_{G_0}-2)u=\|Q\|_{G_0}^2+\frac{1}{6}(\Delta_{G_0}-2)\|Q\|_{G_0}^2=\frac{1}{6}(\Delta_{G_0}+4)\|Q\|_{G_0}^2\geq 0,
\]
where the last inequality follows from Lemma \ref{Lem:Qholomorphic} (ii). 
Suppose ${\rm max}_{p\in S}\,u(p)>0$. Then, at any maximum point $p_0$ of $u$, we have $\Delta_{G_0} u(p_0)<0$ and $-2u(p_0)<0$, which leads $(\Delta_{G_0}-2)u(p_0)<0$, a contradiction. 
Hence $u\leq 0$ on $S$, as required.

Combining \eqref{Lem: Wolf} with \eqref{svff}, we have
\begin{align}\label{hesq}
\Hess_{\Ec_{\Cc}}(Q,Q)_{G_0}\geq \sum_{e\in E}m_E(e)\int_0^1\Big\{\frac{1}{24}\|Q\|^2_{G_0}+\|V_e^{\perp}\|_{G_0}^2+\|\mathcal{V}_e\|_{G_0}^2\Big\}\cdot \|T_e\|_{G_0}^2\, dt\geq 0
\end{align}
for any Weil-Petersson geodesic $\{[G_s]\}_{s \in (-\e, \e)}$ in $\mathcal{T}(S)$, and 
thus $\mathcal{E}_{\mathcal{C}}$ is convex.

Suppose there exists $Q$ such that $\Hess_{\Ec_{\Cc}}(Q,Q)_{G_0}=0$.  Then \eqref{hesq} implies 
\begin{align}\label{hol}
Q(h_{0, e}(t))\equiv 0
\end{align}
for any edge $e\in E$ with $\|T_e\|_{G_0}\neq 0$. For such an edge $e\in E$, we take a point $p$ on the image of $h_e$ and an isothermal coordinate $z=x^1+ix^2$ around $p$.  Recall that, by Lemma \ref{Lem:Qholomorphic}, the function
$Q_{11}-\sqrt{-1}Q_{12}$ is a holomorphic function for $Q\in T_{[G]}\mathcal{T}(S)$. Since  the zeros of a holomorphic function are isolated, \eqref{hol} implies $Q\equiv 0$ on an open subset around the point $p$. 
Then, it turns out that $Q\equiv 0$ on the whole $S$ because $S$ is a (compact) connected surface. 
Therefore for any nonzero $Q\in T_{[G]}\mathcal{T}(S)$, we have $\Hess_{\Ec_{\Cc}}(Q,Q)_{G_0}>0$, and
thus $\mathcal{E}_{\mathcal{C}}$ is strictly convex.
\qed

\begin{remark}
In fact, the above proof shows that in the setting of Theorem \ref{Thm:convex},
the Hessian of the energy functional $\mathcal{E}_{\mathcal{C}}: \mathcal{T}(S)\to \mathbb{R}$ is positive definite and non-degenerate at any point $[G]\in \mathcal{T}(S)$.
\end{remark}

\subsection{Proof of Proposition \ref{key2}} 
Let $\{[G_s]\}_{s \in (-\e, \e)}$ be a smooth curve in $\Tc(S)$, and  $h_{s}$ the family of harmonic maps associated with its horizontal lift $\{G_s\}_{s \in (-\e, \e)}$ in $\Mc_{-1}(S)$.
For the simplicity of notation, let us write
\begin{align*}
T_e^{s}(t):=\p_t h_{s, e}(t), \quad \text{and} \quad V_e^{s}(t):=\p_s h_{s, e}(t).
\end{align*}
for each edge $e\in E$. Moreover, we simply write $T_e(t):=T_e^{0}(t)$ at $s=0$. 
Since $h_{s}: X\to (S, G_{s})$ is a family of harmonic maps, we may assume the following:  
\begin{itemize}
\item $h_{s, e}(t)$ is a constant speed geodesic in $(S, G_{s})$ for each $e\in E$, namely,  we have 
\begin{align}\label{as21}
\nabla^{s}_{T_e^{s}}T_e^{s}=0
\end{align}
along $h_{s, e}(t)$ for any $s \in (-\epsilon, \epsilon)$, where $\nabla^{s}$ is the Levi-Civita connection of $(S, G_{s})$. 
We simply write $\nabla=\nabla^0$.
The parameter $t$ is proportional to the arc-length parameter of $h_{s, e}(t)$, i.e., the norm $\|T_e^{s}(t)\|_{G_{s}}$ is constant along $h_{s, e}(t)$ in $t \in [0, 1]$ for each $e\in E$ and each $s \in (-\epsilon, \epsilon)$.
\item The balanced condition holds for the map $h_{s}: X\to (S, G_{s})$,
\begin{align}\label{as22}
\sum_{e\in E_x}m_E(e) T_e^{s}(0)=0
\end{align}
for any $x\in V$ and all $s\in (-\epsilon, \epsilon)$.
\end{itemize}

 The second variation formula \eqref{Eq:secondVF} in Appendix implies that 
  \begin{align}\label{hf1}
\frac{d^2}{d s^2} E_{G_0}(h_{s})\Big{|}_{s=0}&=\sum_{e\in E}m_E(e) \int_0^1 \Big\{\|\nabla_{T_e}V_e\|_{G_0}^2 -G_0(R(V_e, T_e)T_e, V_e)\Big\}\,dt
\end{align}
 for the harmonic map $h_{0}: X\to (S,G_0)$, where $R$ is the curvature tensor of $(S, G_0)$.
Then, the right hand side becomes
\begin{align}\label{hf2}
&\sum_{e\in E}m_E(e)\int_0^1\Big{\{}-G_0( \nabla_{T_e}\nabla_{T_e}V_e, V_e)-G_0(R(V_e, T_e)T_e, V_e)\Big{\}}\,dt\nonumber\\
&\qquad \qquad \qquad \qquad \qquad \qquad \qquad \qquad \qquad+\sum_{e\in E}m_E(e)\Big[G_0(\nabla_{T_e}V_e, V_e)\Big]_{t=0}^{t=1}.
\end{align}
The following lemma is the place where we use \eqref{as22} that the balanced condition holds for all $s \in (-\e, \e)$.
%The balanced condition \eqref{as22} implies the following:

\begin{lemma} \label{le2}
Letting $T_e^{s}(t):=\p_t h_{s, e}(t)$ and $V_e^{s}(t):=\p_s h_{s, e}(t)$,
if we have
\begin{equation}\label{Eq:BCall-s}
\sum_{e \in E_x}m_E(e)T_e^s(0)=0 \quad \text{for each $x \in V$ for all $s \in (-\e, \e)$},
\end{equation}
then 
$\nabla_{V_x}\cdot \(\sum_{e \in E_x}m_E(e)T_e\)=0$ at each $x \in V$.
In particular, it holds that
\begin{align*}
\sum_{e\in E}m_E(e)\Big[G_0(\nabla_{T_e}V_e, V_e)\Big]_{t=0}^{t=1}=0.
\end{align*}
\end{lemma}
\begin{proof}
For each vertex $x$, differentiating \eqref{Eq:BCall-s} in $s \in (-\e, \e)$, we have
\[
\sum_{e \in E_x}m_E(e)\nabla_{V_x}T_e=\nabla_{V_x}\cdot \(\sum_{e \in E_x}m_E(e)T_e\)=0.
\]
Denote the origin and the terminal of edge $e\in E$ by $o(e)$ and $t(e)$, respectively. 
Since
\[
G_0(V_e, \nabla_{T_e}V_e)_{h_0(t(e))}=-G_0(V_{\overline{e}}, \nabla_{T_{\overline{e}}}V_{\overline{e}})_{h_0(o(\overline{e}))}
\quad \text{for $T_{\overline{e}}=-T_e$ and $m_E(\overline{e})=m_E(e)$},
\]
we have that
\begin{align*}\label{e1}
\sum_{e\in E}m_E(e)\Big{[}G_0(V_e, \nabla_{T_e}V_e)\Big]_{t=0}^{t=1}
=-2\sum_{e\in E}m_E(e)G_0(V_e, \nabla_{T_e}V_e)_{h_0(o(e))}.
%=-2\sum_{x \in V}\sum_{e \in E_x}m_E(e)G_0(V_e, \nabla_{T_e}V_e)_{h_0(x)}
\end{align*}
Note that $[V_e, T_e]=0$ and thus $\nabla_{T_e}V_e=\nabla_{V_e}T_e$.
For each vertex $x$, letting $V_x=V_e$ for $e \in E_x$ we have
\[
-2\sum_{x \in V}\sum_{e \in E_x}m_E(e)G_0(V_e, \nabla_{T_e}V_e)_{h_0(x)}
=-2\sum_{x \in V}G_0(V_x, \sum_{e \in E_x}m_E(e)\nabla_{V_x}T_e)_{h_0(x)}=0,
\]
as required.
%$\sum_{e \in E_x}m_E(e)G_0(V_e, \nabla_{T_e}V_e)_{h_0(x)}=\sum_{e \in E_x}m_E(e)G_0(V_x, \nabla_{V_x}T_e)_{h_0(x)}$
%On the other hand, the balanced condition \eqref{as22} implies that
%\begin{align}\label{c1}
%0=\sum_{e\in E}m_E(e)G_0(V_e^{s}, T_e^{s})_{h_s(o(e))}
%\end{align}
%for any $s$ since $V_e^{s}(h_s(x))$ is independent of the choice of $e\in E_{x}$.  
%Differentiating \eqref{c1} with respect to $s$ at $s=0$, we see
%\begin{align*}
%0&=\sum_{e\in E}m_E(e)G_0(\nabla_{V_e}V_e^{s}|_{s=0}, T_e)_{h_0(o(e))}+\sum_{e\in E}m_E(e)G_0(V_e, \nabla_{V_e}T_e^{s}|_{s=0})_{h_0(o(e))}\\
%&=\sum_{e\in E}m_E(e)G_0(V_e, \nabla_{T_e}V_e)_{h_0(o(e))},\nonumber
%\end{align*}
%where the second equality is due to the balanced condition of $h_0$, the independence of $\nabla_{V_e}V_e^{s}|_{s=0}(x)$ of $e\in E_{x}$ and $[V_e^{s}, T_e^{s}]=0$. 
%Combining this with \eqref{e1}, we obtain the lemma.
\end{proof}

Recall that, for the velocity vector $Q:=(\p_sG_s)|_{s=0}$, 
we have defined the vector field $Q_e^{\sharp}$ along $h_{0, e}(t)$ such that $Q(T_e, \cdot)=G_0(Q_e^{\sharp},\cdot)$ (Definition \ref{Def:V}). 
The following formula uses the property \eqref{as21} that $h_{s, e}(t)$ is a constant speed geodesic in $(S, G_s)$ for each $e \in E$ for all $s \in (-\e, \e)$.

\begin{lemma}\label{le3}
For any horizontal lift $\{G_s\}_{s \in (-\e, \e)}$ of any smooth curve $\{[G_s]\}_{s \in (-\e, \e)}$ for $\e>0$ in $\Tc(S)$, we have
\begin{align*}
-\nabla_{T_e}\nabla_{T_e}V_e-R(V_e, T_e)T_e=\frac{1}{2}\nabla_{T_e}Q_{e}^{\sharp} \qquad \text{for $e \in E$}.
\end{align*}
\end{lemma}

\begin{proof}
We shall consider the differential of  \eqref{as21} at $s=0$.  
For any fixed $t \in [0, 1]$, we may take a local isothermal coordinate $(x^1,x^2)$ around $h_{0,e}(t)$ satisfying  $G_0(h_{0,e}(t))=4((dx^1)^2+(dx^2)^2)$ and $\Gamma_{\beta\gamma}^{\alpha}(G_0, h_{0, e}(t))=0$ at the fixed point $h_{0,e}(t)$. Indeed, since $(S,G_0)$ is a hyperbolic surface,  the fixed point $h_{0,e}(t)$ can be regarded as the origin of the Poincar\'e disk, and the standard coordinate $(x^1,x^2)$ of the Poincar\'e disk gives rise to the desired local coordinate.  

Let
\[
T_e^s(t)=\p_t h_{s, e}(t)=\tau^\a_s(t)\frac{\p}{\p x^\a}\Big{|}_{h_{0, e}(t)} \qquad \text{for $\a=1, 2$}.
\]
Then, we have for $\alpha=1,2$,
\begin{align}\label{c4}
&\frac{\p}{\p s}\Big{|}_{s=0}(\nabla^s_{T_e^s}T_e^s)^{\alpha}
=\frac{\p}{\p s}\Big{|}_{s=0}\Big{(}\p_t \t_s^{\alpha}+\t_s^{\beta}\t_s^{\gamma}\Gamma_{\beta\gamma}^{\alpha}(G_s, h_s)\Big{)}\nonumber\\
&=\frac{\p}{\p s}\Big{|}_{s=0}\Big{(}\p_t \t_s^{\alpha}\Big{)}+\Big(\frac{\p}{\p s}\Big{|}_{s=0}\t_s^{\beta}\t_s^{\gamma}\Big)\Gamma_{\beta\gamma}^{\alpha}(G_0, h_0)\nonumber\\
&\qquad \qquad \qquad \qquad \qquad +\t_0^{\beta}\t_0^{\gamma}\Big(\frac{\p}{\p s}\Big{|}_{s=0}\Gamma_{\beta\gamma}^{\alpha}(G_s, h_0)+\frac{\p}{\p s}\Big{|}_{s=0}\Gamma_{\beta\gamma}^{\alpha}(G_0, h_s)\Big)\nonumber\\
&=\frac{\p}{\p s}\Big{|}_{s=0}\Big(\p_t \t_s^{\alpha}+\t_s^{\beta}\t_s^{\gamma}\Gamma_{\beta\gamma}^{\alpha}(G_0,h_s)\Big)+\frac{\p}{\p s}\Big{|}_{s=0}\Big(\p_t \t_0^{\alpha}+\t_0^{\beta}\t_0^{\gamma}\Gamma_{\beta\gamma}^{\alpha}(G_s, h_0)\Big)\nonumber\\
&=\frac{\p}{\p s}\Big{|}_{s=0}(\nabla^0_{T_e^s}T_e^s)^{\alpha}+\frac{\p}{\p s}\Big{|}_{s=0}(\nabla^s_{T_e}T_e)^{\alpha},
\end{align}
where we have added $\frac{\p}{\p s}\Big{|}_{s=0}(\p_t \t_0^{\alpha})=0$ in the third equality, $\Gamma_{\beta\gamma}^{\alpha}(G_s, h_s)$ denotes the Christoffel symbol for the metric $G_s$ at the point $h_{s, e}(t)$ (and we use the Einstein convention for summations over $\b, \g=1, 2$).  
Since $\Gamma_{\beta\gamma}^{\alpha}(G_0, h_{0, e}(t))=0$, we have
\[
\frac{\p}{\p s}\Big{|}_{s=0}(\nabla^0_{T_e^s}T_e^s)^{\alpha}=(\nabla_{V_e^s}\nabla_{T_e^s}T_e^s)^\a|_{s=0},
\]
and since $[V_e^s, T_e^s]=0$,
\begin{align*}%\label{c41}
(\nabla_{V_e^s}\nabla_{T_e^s}T_e^s)|_{s=0}
=\(R(V_e^s, T_e^s)T_e^s+\nabla_{T_e^s}\nabla_{V_e^s}T_e^s\)|_{s=0}
=\(R(V_e^s, T_e^s)T_e^s+\nabla_{T_e^s}\nabla_{T_e^s}V_e^s\)|_{s=0}
\end{align*}
Therefore we obtain for $\a=1, 2$,
\begin{align}\label{c41}
\frac{\p}{\p s}\Big{|}_{s=0}(\nabla^0_{T_e^s}T_e^s)^{\alpha}=\(R(V_e, T_e)T_e+\nabla_{T_e}\nabla_{T_e}V_e\)^\a.
\end{align}

Consider the second term in \eqref{c4}. 
We simply write $\t^i=\t^i_0$ and $\Gamma_{\beta\gamma}^{\alpha}(s)=\Gamma_{\beta\gamma}^{\alpha}(G_s, h_0)$ so that
\begin{align*}
\nabla^s_{T_e}T_e=\Big{\{}\p_t\tau^\a+ \tau^\b \tau^\g\Gamma_{\b \g}^\a(s)\Big{\}}\frac{\p}{\p x^\a}\Big{|}_{h_{0, e}(t)}
\end{align*}
Since $\nabla_{T_e}T_e$ and $\Gamma_{\b \g}^\a(0)=0$ at $t$, we have
$\p_t\tau^\a(t)=0$ for $\a=1,2$.
Concerning the differential of Christoffel symbol $\Gamma_{ij}^k(s)$ with respect to $s$, we have
\begin{align*}%\label{c420}
\frac{\p}{\p s}\Gamma_{ij}^k(s)\Big{|}_{s=0}&=\frac{\p}{\p s}\frac{1}{2}(G_s)^{km}\{(G_s)_{im,j}+(G_s)_{jm,i}-(G_s)_{ij,m}\}\Big{|}_{s=0}\\
&=\frac{1}{8}\{(\p_s G_s)_{ik,j}+(\p_s G_s)_{jk,i}-(\p_s G_s)_{ij,k}\}\Big{|}_{s=0}\nonumber
\end{align*}
at the point $h_{0, e}(t)$ since $(G_0)^{km}(t)=(1/4)\delta_{km}$ and $(G_0)_{ij,k}(t)=0$ for any $i,j,k=1, 2$.
Recall that $\{G_s\}_{s \in (-\e, \e)}$ is a horizontal lift of a curve $\{[G_s]\}_{s \in (-\e, \e)}$ in $\mathcal{T}(S)$;
Lemma \ref{Lem:Qholomorphic} implies that $(\p_s G_s)_{11}-\sqrt{-1}(\p_s G_s)_{12}$ satisfies the Cauchy-Riemann equation around $h_{0, e}(t)$ and that
 $(\p_s G_s)_{ij,k}|_{s=0}$
 is symmetric under permuting any indices $i, j, k \in \{1, 2\}$. 
Hence we have
\begin{align*}
\frac{\p}{\p s}\Gamma_{ij}^k(s)\Big{|}_{s=0}
=\frac{1}{8}\frac{\p}{\p s}({G}_s)_{ik,j}\Big{|}_{s=0} \quad \text{at $h_{0, e}(t)$}.
\end{align*}
Therefore we see that
\begin{align*}%\label{c42}
\frac{\p}{\p s}\Big{|}_{s=0}(\nabla^s_{T_e}T_e)^{\alpha}
&=\frac{\p}{\p s}\Big{\{}\p_t\tau^\a+\tau^\b \tau^\g\Gamma_{\b \g}^\a(s)\Big{\}}\Big{|}_{s=0}\\
&=\tau^\b \tau^\g \Big(\frac{\p}{\p s}\Gamma_{\b \g}^\a(s)\Big)\Big{|}_{s=0}
=\frac{1}{8}\tau^\b \tau^\g\frac{\p}{\p s}({G}_s)_{\b\a, \g}\Big{|}_{s=0}\nonumber,
\end{align*}
and thus
\begin{align}\label{c42}
\frac{\p}{\p s}\Big{|}_{s=0}(\nabla^s_{T_e}T_e)^{\alpha}
&=\frac{1}{8}\frac{\p}{\p t}\Big\{\frac{\p}{\p s}{G}_s\Big{|}_{s=0}\(T_e, \frac{\p}{\p x^\a}\)\Big\}=\frac{1}{8}\frac{\p}{\p t}Q_e\(T_e, \frac{\p}{\p x^\a}\)\nonumber\\
&=\frac{1}{8}\frac{\p}{\p t}G_0\(Q_e^{\sharp}, \frac{\p}{\p x^\a}\)=\frac{1}{8}G_0\(\nabla_{T_e} Q_e^{\sharp}, \frac{\p}{\p x^\a}\)=\frac{1}{2}(\nabla_{T_e}Q_e^{\sharp})^{\alpha}
\end{align}
where $\(\nabla_{T_e}\frac{\p}{\p x^\a}\)(h_{0, e}(t))=0$ and $G_0(h_{0,e}(t))=4((dx^1)^2+(dx^2)^2)$ in this coordinate.

Finally, substituting \eqref{c41} and \eqref{c42} to \eqref{c4}, we obtain for each $t \in [0, 1]$,
\[
\frac{\p}{\p s}\Big{|}_{s=0}(\nabla^s_{T_e^s}T_e^s)=\nabla_{T_e}\nabla_{T_e}V_e+R(V_e, T_e)T_e+\frac{1}{2}\nabla_{T_e}Q_e^{\sharp},
\]
and since we have \eqref{as21}, we obtain the lemma.
\end{proof}

\proof[Proof of Proposition \ref{key2}]

By Lemma \ref{le2} and 
Lemma \ref{le3}, the consequence of the second variation formula \eqref{hf2} becomes
\begin{align*}
&\frac{d^2}{d s^2} E_{G_0}(h_{s})\Big{|}_{s=0}=\sum_{e\in E}m_E(e)\int_0^1 \frac{1}{2}G_0(\nabla_{T_e}Q_{e}^{\sharp},V_e) \,dt\\
&\qquad \qquad= \frac{1}{2}\sum_{e\in E}m_E(e)\int_0^1-G_0(Q_e^{\sharp},\nabla_{T_e}V_e)\,dt+\frac{1}{2}\sum_{e\in E}m_E(e)\Big[Q(T_e,V_e)\Big]_{t=0}^{t=1}. 
\end{align*}
The second term is $0$ since the balanced condition \eqref{as22} implies that
\begin{align*}
\frac{1}{2}\sum_{e\in E}m_E(e)\Big[Q(T_e,V_e)\Big]_{t=0}^{t=1}
&=-\sum_{e\in E}m_E(e)Q({T_e}(0), V_e(0))\\
&=-Q\(\sum_{e\in E}m_E(e)T_e(0), V_e(0)\)=0.
\end{align*}
The integrand of the first term for each $e \in E$ is 
\begin{align}\label{svf3}
-G_0(Q_e^{\sharp}, \nabla_{T_e}V_e)
=\|\nabla_{T_e}V_e\|_{G_0}^2+\frac{\|Q_e^{\sharp}\|_{G_0}^2}{4}-\Big\|\nabla_{T_e}V_e+\frac{Q_e^{\sharp}}{2}\Big\|^2_{G_0}.
\end{align}
Letting $J_0$ be the complex structure on $(S, G_0)$, we have for $e \in E$ with $T_e \neq 0$,
\begin{align*}
\|Q_e^{\sharp}\|_{G_0}^2=\frac{1}{{\|T_e\|_{G_0}^2}}\{{Q(T_e,T_e)^2}+{Q(T_e, J_0 \cdot T_e)^2}\}
=\frac{\|Q\|_{G_0}^2\|T_e\|_{G_0}^2}{2},
\end{align*}
and note that by Definition \ref{Def:V} for $e \in E$ with $T_e \neq 0$,
\[
\mathcal{V}_e=\frac{1}{\|T_e\|_{G_0}}\(\nabla_{T_e}V_e+\frac{Q_e^{\sharp}}{2}\),
\]
and
by \eqref{hf2},
\begin{align*}
\sum_{e\in E}m_E(e)\int_0^1\|\nabla_{T_e}V_e\|_{G_0}^2\, dt
&=\frac{d^2}{ds^2}E_{G_0}(h_s)\Big{|}_{s=0}+\sum_{e\in E}m_E(e)\int_0^1 G_0(R(V_e, T_e)T_e,V_e)\, dt.
\end{align*}
Substituting these to \eqref{svf3}, we obtain
\begin{align*}
&\frac{1}{2}\cdot \frac{d^2}{d s^2} E_{G_0}(h_{s})\Big{|}_{s=0}\\
&=\frac{1}{2}\sum_{e\in E}m_E(e)\int_0^1\Big{\{}G_0(R(V_e, T_e)T_e,V_e)+\frac{\|Q\|_{G_0}^2\|T_e\|_{G_0}^2}{8}-\|\mathcal{V}_e\|_{G_0}^2 \|T_e\|_{G_0}^2\Big{\}}\, dt.
\end{align*}
Since $(S,G_0)$ is a surface of constant sectional curvature $-1$, we have
\begin{align*}
G_0(R(V_e, T_e)T_e,V_e)=-\|V_e^{\perp}\|^2_{G_0}\|T_e\|_{G_0}^2,
\end{align*}
where $V_e^{\perp}$ is the orthonormal projection with respect to $T_e$ for $T_e=0$ and $0$ otherwise,
and we complete the proof of Proposition \ref{key2}.
\qed

\section{Properness of the energy functional}\label{Sec:proper}

We will establish the properness of the energy functional $\Ec_\Cc$ for $\Cc=[f]$ if $f:X \to S$ induces a surjective homomorphism between the fundamental groups.
First we show the following lemma (which will also be used in the next Section \ref{Sec:MCG}).

\begin{lemma}\label{Lem:homotopic}
Assume that $f:X \to S$ is a continuous map which induces a surjective homomorphism from $\pi_1(X)$ to $\pi_1(S)$.
If $\f_1, \f_2$ are orientation-preserving homeomorphisms on $S$ and $\f_1 \circ f$ and $\f_2 \circ f$ are homotopic, 
then $\f_1$ and $\f_2$ are isotopic on $S$.
\end{lemma}

\proof
Let $\wh \f:=\f_2^{-1}\circ \f_1$,
then
the maps $\wh \f \circ f$ and $f$ are homotopic.
For any $x_0 \in X$, let $\g$ be the path from $f(x_0)$ to $\wh \f \circ f(x_0)$ given by the homotopy between $\wh \f \circ f$ and $f$.
Then, there exists a homeomorphism $\b_\g$ homotopic to the identity map $\id$ on $S$ such that
\[
\b_{\g \ast}:\pi_1(S, \wh \f \circ f(x_0)) \to \pi_1(S, f(x_0)), \qquad \b_{\g \ast}([\a])=[\g \cdot \a \cdot \wbar \g],
\]
for $[\a] \in \pi_1(S, \wh \f \circ f(x_0))$, where $\wbar \g$ denotes the reversed path of $\g$. 
(Such an $\b_\g$ is obtained first when $\g$ is a short enough path on $S$, and in general divide $\g$ into finitely many short enough paths and composite homeomorphisms homotopic to the identity.)
Since $f$ induces a surjective homomorphism $f_\ast: \pi_1(X, x_0) \to \pi_1(S, f(x_0))$,
one obtains
$\b_{\g \ast} \circ \wh \f_\ast: \pi_1(S, f(x_0)) \to \pi_1(S, f(x_0))$, 
\[
[f(c)] \mapsto [\g \cdot \wh \f (f(c)) \cdot \wbar \g] \quad \text{for $[c] \in \pi_1(X, x_0)$},
\]
and since $\wh \f \circ f$ and $f$ are homotopic,
$\b_{\g \ast} \circ \wh \f_\ast$ coincides with $\id_\ast$ given by $\id$ on $S$.
Since $S$ is a $K(\pi_1(S), 1)$-space, $\b_\g \circ \widehat \f$ and $\id$ are homotopic fixing $f(x_0)$ \cite[Proposition 1B.9]{Hatcher}, and since $\b_\g$ is homotopic to $\id$,
$\widehat \f$ is homotopic to $\id$; hence
$\f_1$ and $\f_2$ are homotopic.
It is known that two orientation-preserving homeomorphisms which are homotopic on a closed Riemann surface $S$ are isotopic \cite[Theorem 6.3]{Epstein-isotopies}, therefore $\f_1$ and $\f_2$ are in fact isotopic on $S$.
\qed

\def\Sub{{\rm Sub}}

\proof[Proof of Theorem \ref{Thm:proper}]
For any $R \ge 0$, let us denote the sublevel set by
\[
\Sub(R):=\Big{\{}[(\SS, G_\SS), \f] \in \Tc(S) \ : \ \Ec_\Cc([G]) \le R, \, G=\f^\ast G_\SS\Big{\}}.
\]
Assume that $\Sub(R) \neq \emptyset$.
Taking a harmonic map $h_G: X \to (S, G)$ in the homotopy class $\Cc=[f]$, we write
\[
\Ec_\Cc([G])=E_G(h_G)=\frac{1}{2}\sum_{e \in E}m_E(e)\int_0^1\|\partial_t h_{G, e}\|_G^2\, dt.
\]
Let $l(h_{G, e})$ be the length of $h_{G, e}:[0, 1] \to (S, G)$ for each $e \in E$.
Since $h_G$ is harmonic and piecewise geodesic,  
we have for each $e \in E$,
\[
l(h_{G, e})=\int_0^1\|\partial_t h_{G, e}\|_G\, dt=\|\partial_t h_{G, e}\|_G.
\]
The Cauchy-Schwarz inequality gives 
\begin{align*}
\sum_{e \in E}l(h_{G, e})
&= \sum_{e \in E}\(\int_0^1 \|\partial_t h_{G, e}\|_G^2 \, dt\)^{1/2} \\
&\le \(\sum_{e \in E}\frac{1}{m_E(e)}\)^{1/2}\cdot \(\sum_{e \in E}m_E(e)\int_0^1 \|\partial_t h_{G, e}\|_G^2 \, dt\)^{1/2}.
\end{align*}
Hence letting $M:=\sum_{e \in E}m_E(e)^{-1}$, we obtain for any $[G] \in \Sub(R)$,
\begin{equation}\label{Eq:length}
\sum_{e \in E}l(h_{G, e}) \le \(M\cdot 2\Ec_{\Cc}([G])\)^{1/2} \le (2 MR)^{1/2}.
\end{equation}

For given closed Riemann surface $S$, we take a finite collection of simple 
closed curves $\{\g_1, \dots, \g_N\}$ on $S$ such that the union $\g_1 \cup \cdots \cup \g_N$ fills $S$, i.e., the complement of this union is a disjoint union of disks.
(For example, one may take simple closed curves corresponding to the free homotopy classes of the standard set of generators of $\pi_1(S)$, where $N$ is twice the genus of $S$.) 
Then, if $\g$ is any simple closed curve which is not homotopic to a point, then there exists a $\g_i$ such that the geometric intersection number between free homotopy classes of $\g$ and $\g_i$ is not zero.

Since $f_\ast: \pi_1(X, x_0) \to \pi_1(S, f(x_0))$ is surjective, we choose a collection of closed paths $\{c_1, \dots, c_N\}$ in $X$ such that $f(c_i)$ is freely homotopic to $\g_i$ for each $i=1, \dots, N$.
Let $C_{\max}:=\max_{i=1, \dots, N} |c_i|$, where $|c_i|$ is the number of edges in $c_i$.
(One may just take $C_{\max}=1$ if $f$ fills the surface in the following inequality.)
Then, for any $c_i$, we have
\[
\sum_{e \in c_i}l(h_{G, e}) \le C_{\max}\, \sum_{e \in E}l(h_{G, e}) \le C_{\max}\,(2 MR)^{1/2},
\]
for any $[G] \in \Sub(R)$.
For any closed hyperbolic surface $(S, G)$, let $\inj((S, G))$ be the injectivity radius.
Taking a closed geodesic $\g_{\inj}$ in $(S, G)$ such that the length realizes twice the injectivity radius $l(\g_{\inj})=2\, \inj ((S, G))$,
we have a $\g_i$ such that the geometric intersection number is non zero in their free homotopy classes.
Let $\g_{i, G}$ be a (unique) closed geodesic freely homotopic to $\g_i$, 
then $\g_{i, G}$ has the shortest length among the free homotopy class of $\g_i$ since $(S, G)$ is a hyperbolic surface,
and we have
\[
l(\g_{i, G}) \le \sum_{e \in c_i}l(h_{G, e}) \le C_{\max}\,(2 MR)^{1/2}.
\]
Then, the Collar Lemma \cite{Keen}
(e.g., \cite[Lemma 13.6]{FarbMargalit})
implies that
\[
N_{\g_i}=\left\{x \in S \ : \ d(x, \g_i) \le \sinh^{-1}\(\frac{1}{\sinh (l(\g_i)/2)}\)\right\}
\]
is an embedded annulus, and thus
\[
l(\g_{\rm inj}) \ge \sinh^{-1}\(\frac{1}{\sinh (l(\g_i)/2)}\).
\]
Therefore, letting
\[
\e(R):=\frac{1}{2}\,\sinh^{-1}\(\frac{1}{\sinh ((1/2)C_{\max}\,(2 MR)^{1/2})}\)>0,
\]
we obtain
$\inj ((S, G)) \ge \e(R)$ for any $[G] \in \Sub(R)$.
We consider the moduli space $\Mc(S)$ of $S$, and define
\[
\Mc_{\e(R)}(S):=\Big{\{}[(S, G)] \in \Mc(S) \ : \ \inj((S, G))\ge \e(R) \Big{\}}.
\]
Now we are discussing the standard topology in Teichm\"uller space $\Tc(S)$, and the induced topology in the moduli space $\Mc(S)$ as the quotient space of $\Tc(S)$ by the action of the mapping class group $\Mod(S)$ which acts on $\Tc(S)$ properly discontinuously. 
Then, the Mumford-Mahler Compactness criterion
implies that $\Mc_{\e(R)}(S)$ is compact (\cite[Theorem C.1]{TrombaBook});
for any sequence of points $[(\SS_i, G_i), \f_i]$ in $\Tc(S)$ with $\Ec_\Cc([\f_i^\ast G_i]) \le R$, we have $\inj((S, \f_i^\ast G_i)) \ge \e(R)$, and thus up to passing to a subsequence there exists a hyperbolic surface $[(S, G_\star)]$ in $\Mc_{\e(R)}(S)$ such that 
$[(S, \f_i^\ast G_i)]$ converges to $[(S, G_\star)]$ in $\Mc(S)$.
Moreover, there exists a sequence of (orientation-preserving) diffeomorphisms $\wh \f_i$ on $S$ to itself and $\f_\star:=\id:S \to S$ such that $[(\SS_i, G_i), \f_i \circ \wh \f_i^{-1}]$ converges to $[(S, G_\star), \f_\star]$ in the Teichm\"uller space $\Tc(S)$, and
$(\f_i \circ \wh \f_i^{-1})^\ast G_i$ converges to $G_\star$ in the $C^\infty$-topology:
for any $\e>0$, for all large enough $i$, we have
\begin{equation}\label{Eq:dilatation}
G_\star \le (1+\e)(\f_i \circ \wh \f_i^{-1})^\ast G_i.
\end{equation}

Let $h_i:X \to  (S, \f_i^\ast G_i)$ be a harmonic map in the homotopy class $\Cc$.
The maps $\wh \f_i \circ h_i:X \to (S, (\f_i \circ \wh \f_i^{-1})^\ast G_i)$ are harmonic since $\wh \f_i: (S, \f_i^\ast G_i) \to (S, (\f_i \circ \wh \f_i^{-1})^\ast G_i)$ are isometry
(although $\wh \f_i \circ h_i$ may not be homotopic to $h_i$).
We note that $\wh \f_i \circ h_i: X \to (S, G_\star)$ are equicontinuous; this follows from (\ref{Eq:length}) and (\ref{Eq:dilatation}).
Thus by the Ascoli-Arzel\`a theorem after passing to a subsequence if necessary, $\wh \f_i \circ h_i$ converges to a continuous map $h_\star: X \to (S, G_\star)$ uniformly.
Therefore for all large enough $i, j$, the maps $\wh \f_i \circ h_i$ and $\wh \f_j \circ h_j$ are homotopic.

Then, the condition that $\wh \f_i \circ h_i$ and $\wh \f_j \circ h_j$ are in the same homotopy class inducing surjective homomorphisms $\pi_1(X) \to \pi_1(S)$ implies that $\wh \f_i$ and $\wh \f_j$ are isotopic on $S$.
Indeed, since both $h_i$ and $h_j$ are homotopic to $f$, $\wh \f_i \circ f$ and $\wh \f_j \circ f$ are homotopic for all large enough $i, j$, and
Lemma \ref{Lem:homotopic} implies that $\wh \f_i$ and $\wh \f_j$ are isotopic for all large enough $i, j$.

Let $\wh \f_\star:=\wh \f_i$ for a large enough $i$. Since $[(\SS_i, G_i), \f_i \circ \wh \f_i^{-1}]$ converges to $[(S, G_\star), \f_\star]$, 
\[
[(\SS_i, G_i), \f_i] \to [(S, G_\star), \f_\star \circ \wh \f_\star]=[(S, G_\star), \wh \f_\star] \quad \text{ in $\Tc(S)$}.
\]
For any $\e>0$, and for all large enough $i$, we have \eqref{Eq:dilatation} and
\[
\Ec_\Cc([(\wh \f_\star)^\ast G_\star]) \le E_{(\wh \f_\star)^\ast G_\star}(h_i) \le (1+\e)E_{\f_i^\ast G_i}(h_i),
\]
and conclude that $\Ec_\Cc([(\wh \f_\star)^\ast G_\star]) \le R$ by letting $\e \to 0$, as required.
\qed

\section{Actions by finite subgroups of the mapping class groups}\label{Sec:MCG}

Recall that $\Aut(X)$ is the group of automorphisms $\s$ of a finite weighted graph $X=(V, E, m_E)$,
where $\s$ is a bijection from $V$ to itself, preserves edges with $\wbar{\s e}=\s \wbar e$ and $m_E(\s e)=m_E(e)$ for $e \in E$,
and also recall that the mapping class group $\Mod(S)$ is the group of isotopy classes of orientation-preserving homeomorphisms $\Homeo^+(S)$.
We define a subgroup of $\Mod(S)$ for any continuous map $f:X \to S$ by
\[
\Gc_{[f]}(S):=\Big{\{}[\f] \in \Mod(S) \ : \ \text{$\f \circ f \simeq f \circ \s$ for some $\s \in \Aut(X)$}\Big{\}},
\]
where $f_1 \simeq f_2$ if and only if $f_1$ and $f_2$ are homotopic.
This $\Gc_{[f]}(S)$ is a group depending only on the homotopy class of $f$.

\begin{lemma}\label{Lem:finiteness}
If $f: X \to S$ is a continuous map which induces a surjective homomorphism $\pi_1(X, x_0) \to \pi_1(S, f(x_0))$,
then the group $\Gc_{[f]}(S)$ is a finite group.
\end{lemma}

\proof
Let us define a subgroup of $\Aut(X)$ by
\[
\Gc_{[f]}(X):=\Big{\{}\s \in \Aut(X) \ : \ \text{$\f \circ f \simeq f \circ \s$ for some $[\f] \in \Mod(S)$}\Big{\}}.
\]
Assigning $\s \in \Gc_{[f]}(X)$ to $[\f] \in \Mod(S)$ such that $\f \circ f \simeq f \circ \s$,
we define the map 
\[
\pi: \Gc_{{[f]}}(X) \to \Gc_{[f]}(S).
\]
This map $\pi$ is well-defined; indeed, for $\s \in \Aut(X)$, if there are $\f_1, \f_2 \in \Homeo^+(S)$ such that $\f_1 \circ f \simeq f \circ \s$ and $\f_2 \circ f \simeq f \circ \s$, then $\f_1 \circ f \simeq \f_2 \circ f$.
Since $f:X \to S$ induces a surjective homomorphism from $\pi_1(X)$ to $\pi_1(S)$, Lemma \ref{Lem:homotopic} implies that $\f_1$ and $\f_2$ are isotopic on $S$, i.e., $[\f_1]=[\f_2] \in \Mod(S)$.
The map $\pi$ is a group homomorphism which is surjective by the definitions of $\Gc_{[f]}(S)$ and $\Gc_{[f]}(X)$.
Since the group $\Gc_{[f]}(X)$ is a subgroup of $\Aut(X)$ which is a finite group, and $\Gc_{[f]}(S)=\pi(\Gc_{[f]}(X))$, we conclude that $\Gc_{[f]}(S)$ is a finite group.
\qed

\proof[Proof of Theorem \ref{Thm:Nielsen}]
By Lemma \ref{Lem:fill}, we assume that 
$f$ induces a surjective homomorphism from $\pi_1(X)$ to $\pi_1(S)$.
Then Lemma \ref{Lem:finiteness} implies that $\Gc_{[f]}(S)$ is a finite group.
Let us simply write $\Gc=\Gc_{[f]}(S)$, 
and define
\begin{equation*}%\label{Eq:hatE}
\wh \Ec_\Cc:=\frac{1}{|\Gc|}\sum_{[\f] \in \Gc}\Ec_{\Cc}\circ [\f] : \Tc(S) \to [0, \infty).
\end{equation*}
Since $\Ec_{\Cc}$ is strictly convex with respect to the Weil-Petersson metric as well as proper by Theorem \ref{Thm:convex} and \ref{Thm:proper},
and $\Mod(S)$ acts on $\Tc(S)$ as isometries relative to the Weil-Petersson metric, each function $\wh \Ec_\Cc \circ [\f]$ for $[\f] \in \Gc$, and its convex combination $\wh \Ec_\Cc$ are also so.
Hence $\wh \Ec_\Cc$ has a unique minimizer $[(\wh \SS_0, \wh G_0), \wh \f_0]$ in $\Tc(S)$.
Note that $\wh \Ec_\Cc$ is invariant under the action of $\Gc$, and thus the minimizer is fixed by $\Gc$.
This implies that $\Gc$ acts on $(S, \wh \f_0^\ast \wh G_0)$ as isometries.
Each $[\f] \in \Gc$ has a unique representative as an isometry of $(S, \wh \f_0^\ast \wh G_0)$ since two isometries which are isotopy (or, homotopy) must coincide.
Let us define $\wt \Gc_0$ as the group generated by these isometries.

So far,
we have obtained a unique minimizer $[(\wh \SS_0, \wh G_0), \wh \f_0]$ for $\wh \Ec_\Cc$ and a group $\wt \Gc_0$ of isometries on $(S, \wh \f_0^\ast \wh G_0)$ such that $\wt \Gc_0 \to \Gc_{[f]}(S)$, $\wt \f \mapsto [\wt \f]$ induces an isomorphism of groups.
Then, we shall show that a point in $\Tc(S)$ is the unique minimizer of $\wh \Ec_\Cc$ if and only if it is the unique minimizer of $\Ec_\Cc$.

Suppose that $[(\SS_0, G_0), \f_0]$ is the unique minimizer of $\Ec_\Cc$.
Let $h_0: X \to (S, \f_0^\ast G_0)$ be the unique harmonic map in the homotopy class $\Cc=[f]$.
Then it is fixed by $\Gc_{[f]}(S)$:
indeed, for any $[\f] \in \Gc_{[f]}(S)$, there exists $\s_{[\f]} \in \Aut(X)$ such that
\[
\f \circ h_0 \simeq h_0 \circ \s_{[\f]}.
\]
Take a diffeomorphism $\f$ as a representative of $[\f]$, and a unique harmonic map 
\[
h_0':X \to (S, (\f_0 \circ \f^{-1})^\ast G_0),
\]
such that $h_0'$ is homotopic to $h_0$.
Then, noting that $h_0' \circ \s_{[\f]}$ is the unique harmonic map homotopic to $\f \circ h_0$ from $X$ to $(S, (\f_0 \circ \f^{-1})^\ast G_0)$,
we have
\[
E_{\f_0^\ast G_0}(h_0)=E_{(\f_0\circ \f^{-1})^\ast G_0}(\f \circ h_0) \ge E_{(\f_0\circ \f^{-1})^\ast G_0}(h_0' \circ \s_{[\f]})=E_{(\f_0\circ \f^{-1})^\ast G_0}(h_0'),
\]
where we have used that the energy $E_{G}(h_0)$ is invariant under pre-composing $\s_{[\f]}$ to $h_0$ in the last equality.
Hence
\[
\Ec_\Cc([\f_0^\ast G_0]) \ge \Ec_\Cc([\f]\cdot [\f_0^\ast G_0]),
\]
and the uniqueness of minimizer for $\Ec_\Cc$ implies that $[(\SS_0, G_0), \f_0]=[\f]\cdot [(\SS_0, G_0), \f_0]$ for any $[\f] \in \Gc_{[f]}(S)$.

As a consequence, we have
\[
\wh \Ec_\Cc([G_0])=\Ec_\Cc([G_0]).
\]
On the other hand,
if $[(\wh \SS_0, \wh G_0), \wh \f_0]$ is the unique minimizer of $\wh \Ec_\Cc$, then it is fixed by $\Gc_{[f]}(S)$ since $\wh \Ec_\Cc$ is invariant under $\Gc_{[f]}(S)$, 
and thus
\[
\wh \Ec_\Cc([\wh G_0])=\Ec_\Cc([\wh G_0]).
\]
Therefore $[G_0]=[\wh G_0]$.

Recall that 
the group $\wt \Gc_0$ acts on $(S, \f_0^\ast G_0)$ as isometries.
Since for each $\wt \f \in \wt \Gc_0$, $\wt \f \circ h_0$ is a harmonic map homotopic to $h_0 \circ \s_{[\wt \f]}$ for some $\s_{[\wt \f]} \in \Aut(X)$,
we have
\[
\wt \f \circ h_0=h_0 \circ \s_{[\wt \f]},
\]
by the uniqueness of harmonic map in the same homotopy class by Theorem \ref{Thm:KSharmonic} (iii).
We complete the proof.
\qed

\begin{remark}\label{Rem:Nielsen}
Note that for any finite subgroup $\Gc$ of $\Mod(S)$, there exists a hyperbolic surface $(S, G)$ and a group $\wt \Gc$ of isometries of $(S, G)$ such that the natural map $\wt \Gc \to \Gc$, $\f \mapsto [\f]$ gives an isomorphism of groups.
This is known as the {\it Nielsen realization problem}
(a theorem by Kerckhoff \cite{KerckhoffNielsen}; for the history and a proof, we refer to \cite[Section 6.4]{TrombaBook}, which is regarded as a continuous counterpart to our setting).
The proof which we have given above also provides another proof of this problem by using the functional $\Ec_\Cc$;
note that the first paragraph of the proof of Theorem \ref{Thm:Nielsen} works for any finite subgroup $\Gc$ of $\Mod(S)$.
\end{remark}

\section{Examples}\label{Sec:examples}

We give some examples where we are able to find the unique minimizer of energy functional $\Ec_\Cc$ in a fairly explicit way.
The first class of examples we discuss is a hyperbolic surface constructed by gluing right-angled hexagons in the hyperbolic plane $\H^2$.
Given a specific map $f:X \to S$,
we determine each hexagon and how they are combined to form the hyperbolic surface $(S, G)$ where the harmonic map into $(S, G)$ homotopic to $f$ realizes the minimum of $\Ec_\Cc$ with $\Cc=[f]$.
We discuss the case of closed surface  $S_2$ of genus $2$ with a map $f:X \to S_2$ in Section \ref{Sec:example_genus_2}, and generalize that construction to closed surface of higher genus in Section \ref{Sec:example_genus_g}.
Another class of examples is provided by classical triangle tessellations on the hyperbolic plane $\H^2$.
We discuss general triangle tessellations in Section \ref{Sec:example_triangle}, and in particular
show that Klein's surface of genus $3$ arises as the minimizer for some explicit energy functional.
In all cases, we apply Theorem \ref{Thm:Nielsen}; we make use of the group $\Gc_{[f]}(S)$ which we have introduced in order to reduce the dimension of parameter space (the Teichm\"uller space).

\subsection{A simple example for a closed surface of genus $2$}\label{Sec:example_genus_2}

A closed Riemann surface of genus two $S_2$ is obtained by two pairs of pants $\Pc$ and $\Pc'$.
A pair of pants is topologically a compact connected surface of genus zero with three boundary components which are circles.
Each pair of pants $\Pc$ (resp.\ $\Pc'$) is obtained by two hexagons $\Hc$ and $\wbar \Hc$ (resp.\ $\Hc'$ and $\wbar \Hc'$), where a set of every other sides and a set of corresponding sides are identified.
We glue two pairs of pants together by identifying three pairs of boundary circles, and obtain the surface $S_2$.
When we endow $S_2$ with a hyperbolic metric $G$, we realize each hexagon as an isometric copy of a right-angled hexagon whose sides are geodesics in the hyperbolic plane $\H^2$.
Note that by the Gauss-Bonnet theorem, any right-angled hexagon in $\H^2$ has the area $\pi=3.1415\cdots$.

We define a finite graph whose edges are sides of hexagons in the surface $S_2$, and vertices are the points at intersections of sides.
Let $X=(V, E, m_E)$ be the corresponding weighted graph with unit weight on edges $m_E \equiv 1$.
The map $f$ we consider is the natural embedding of $X$ into $S_2$.
Then, the map $f$ fills the surface $S_2$.
Therefore there exists a unique minimizer $(S_2, G_0)$ for the energy functional $\Ec_\Cc$ with $\Cc=[f]$ by Theorem \ref{Thm:main}.

\def\rot{{\rm rot}}
\def\inv{{\rm inv}}
\def\ex{{\rm ex}}

The group $\Gc_{[f]}(S_2)$ has three elements $\f_\rot$, $\f_\inv$ and $\f_\ex$, where $\f_\rot$ is given by the rotation of order $3$ corresponding to a cyclic permutation of glued three boundary components of two pairs of pants, $\f_\inv$ is given by an involution of order $2$ corresponding to exchanging two pairs of pants, and $\f_\ex$ is given by an involution of order $2$ exchanging hexagons $\Hc$ and $\wbar \Hc$, as well as $\Hc'$ and $\wbar \Hc'$ simultaneously
(Figure \ref{Fig:surface}).
\begin{figure}[h!]
\centering
\includegraphics[width=100mm]{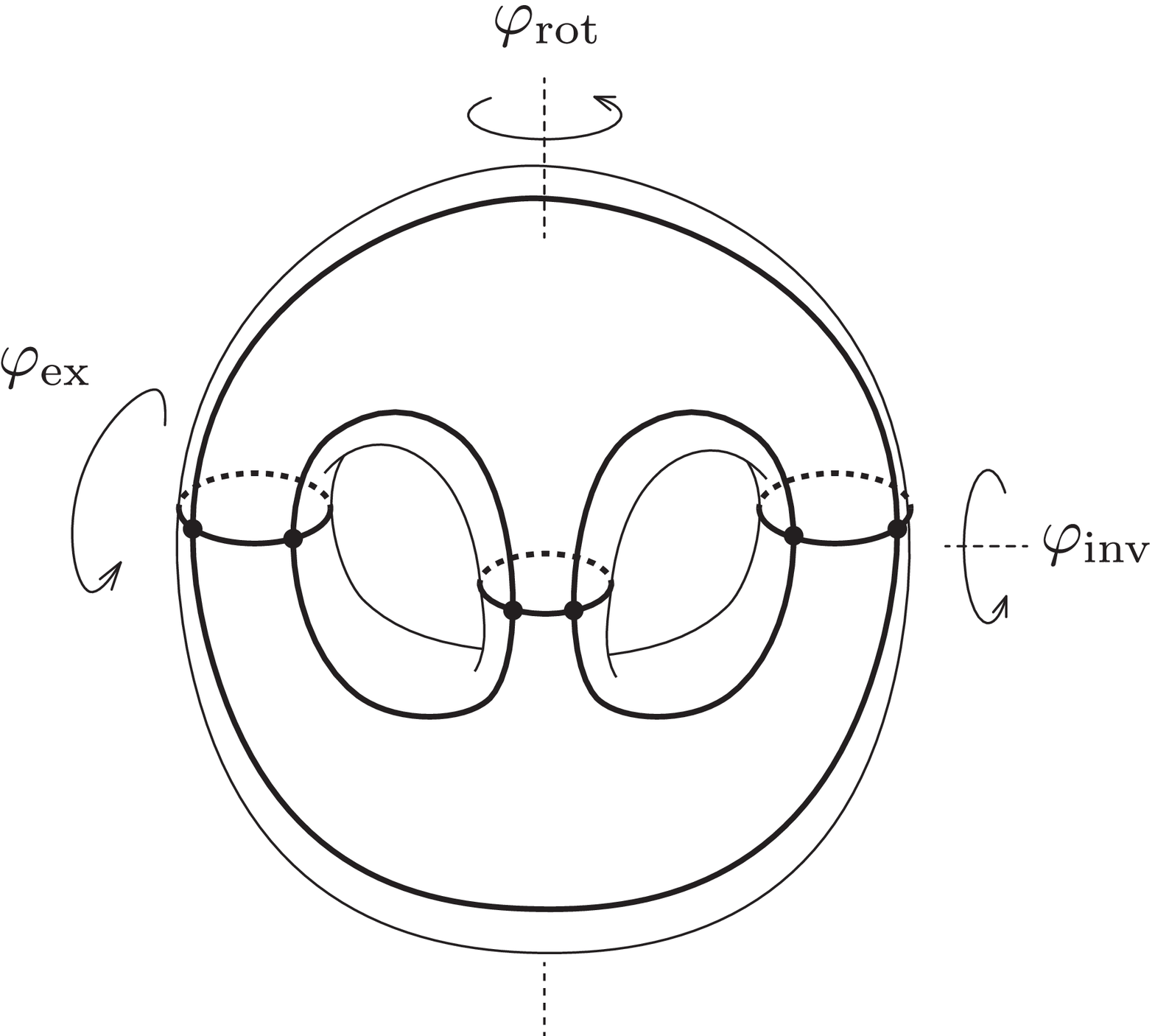}
\caption{A surface of genus $2$ with actions by $\f_\rot$, $\f_\inv$ and $\f_\ex$.}
\label{Fig:surface}
\end{figure}
If a hyperbolic surface $(S_2, G_0)$ realizes the minimum of $\Ec_\Cc$ with a harmonic map $h_0$ in the homotopy class $\Cc=[f]$, then Theorem \ref{Thm:Nielsen} implies that $(S_2, G_0)$ admits isometries corresponding to $\f_\rot$, $\f_\inv$ and $\f_\ex$, and the image of $h_0$ which consists of geodesic segments has to be invariant under their actions (as a set). 
This requires that four hexagons which are complements of the image of $h_0$ are isometric copies of a single right-angled hexagon $\Hc_0$ whose every other sides are geodesic segments of the same length.
Indeed, first each pair of pants $\Pc$ (resp.\ $\Pc'$) has three boundary components which are all (isometric) closed geodesics since $h_0$ is harmonic.
Second, every other sides in each hexagon have the same length by the action of $\f_\rot$, two pairs of pants $\Pc$ and $\Pc'$ are isometric by the action of $\f_\inv$, and the hexagons $\Hc$ and $\wbar \Hc$ (resp.\ $\Hc'$ and $\wbar \Hc'$) are isometric by the action of $\f_\ex$.
Then, the sides of these hexagons in the surface realizes the image of the harmonic map $h_0$.

This process reduces the space of parameters to be determined; originally it has the (real) dimension $6g-6=6$ (the dimension of Teichm\"uller space of $S_2$ where $g=2$), but now it is as large as the space of parameters to determine $\Hc_0$ --- it has the dimension $1$.

Let $\Hc_\ast$ be a right-angled hexagon whose sides are geodesic segments in $\H^2$ with one distinguished vertex (a mark).
We equip $\Hc_\ast$ with the induced orientation from $\H^2$.
From the mark, we denote every other sides in the counter-clockwise by $c_1$, $c_2$ and $c_3$, and remaining sides by $d_1$, $d_2$ and $d_3$ (Figure \ref{Fig:hexagon}).
\begin{figure}[h!]
\centering
\includegraphics[width=70mm]{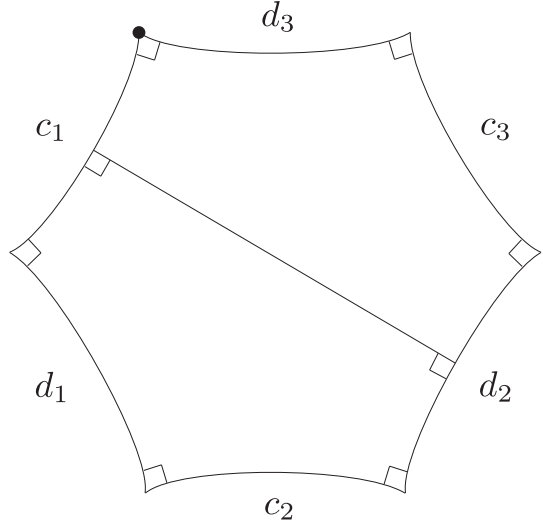}
\caption{A marked right-angled hexagon in the hyperbolic plane.}
\label{Fig:hexagon}
\end{figure}
Let $\wbar \Hc_\ast$ be the hexagon given by a reflection of $\Hc_\ast$. 
We glue $\Hc_\ast$ and $\wbar \Hc_\ast$ along $d_1$, $d_2$ and $d_3$ with their corresponding copies, and obtain a hyperbolic pair of pants $\Pc_0$.
Let $\wt c_1$, $\wt c_2$ and $\wt c_3$ be the boundary circles corresponding to $c_1$, $c_2$ and $c_3$.
Let $\Pc_0'$ be a reflected isometric copy of $\Pc_0$, and we glue $\Pc_0$ and $\Pc_0'$ by paring $\wt c_i$ and its corresponding copy for $i=1, 2, 3$.
The hyperbolic metric on the resulting surface is determined by the hexagon $\Hc_\ast$ whose sides are $c_i$ and $d_i$ for $i=1, 2, 3$.
Let $t:=l(c_i)$ and $s:=l(d_i)$ be the length of $c_i$ and $d_i$ for all $i=1, 2, 3$.
Then we have
\[
t=2 \sinh^{-1}\(\frac{\cosh (s/2)}{\sinh s}\),
\]
by the hyperbolic trigonometry; taking the geodesic from the midpoint in $c_1$ to the midpoint of $d_2$ in $\Hc$, we have in the right-angled pentagon containing $c_2$,
\[
\sinh\(\frac{l(c_1)}{2}\)\cdot \sinh\(l(d_1)\)=\cosh\(\frac{l(d_2)}{2}\)
\]
\cite[(V.2) in Section 8.1]{FenchelNielsenBook}.
Let $(S_2, G_{s})$ be the resulting hyperbolic surface.
If we define the map $h_s: X \to (S_2, G_s)$ so that the image $h_s(X)$ is realized by the sides of hexagons and all $h_{s, e}:[0, 1] \to (S_2, G_s)$ for $e \in E$ are the constant speed geodesics,
then $h_s$ is harmonic since it satisfies the balanced condition on each vertex.
The energy function becomes
\[
\Ec_{\Cc}([G_{s}])=E_{G_{s}}(h_s)=6(s^2+t^2)=6s^2+24\(\sinh^{-1}\(\frac{\cosh (s/2)}{\sinh s}\)\)^2.
\]

In fact, we may extend the above argument slightly to the case when all the weights are not equal to $1$.
If the edges corresponding to the sides $c_i$ of the hexagon $\Hc_\ast$ have the weight $m_c$ and the edges corresponding to the sides $d_i$ of $\Hc_\ast$ have the weight $m_d$ for all $i=1, 2, 3$, then the map $h_s:X \to (S_2, G_s)$ is still harmonic in this case as well.
The energy functional becomes
\[
\Ec_{\Cc}([G_s])=E_{G_{s}}(h_s)=6(m_d\, s^2+m_c\, t^2)=6m_d\, s^2+24m_c\(\sinh^{-1}\(\frac{\cosh (s/2)}{\sinh s}\)\)^2.
\]
Then we shall find the corresponding unique minimizer for the weights $m_c$ and $m_d$.

The problem is to seek the minimizer of the function
\[
m_d\,s^2+m_c\,t^2
\] 
under the restriction
\begin{align}\label{note1}
\sinh\Big{(}\frac{t}{2}\Big{)}\cdot \sinh(s)=\cosh\Big{(}\frac{s}{2}\Big{)},\quad s>0,\quad t>0,
\end{align} 
which is equivalent to
\[
\sinh\Big{(}\frac{s}{2}\Big{)}\sinh\Big{(}\frac{t}{2}\Big{)}=\frac{1}{2}.
\]
Letting $F(s,t):=\sinh(s/2)\sinh(t/2)-1/2$, we have
\[
\p_sF=(1/2)\cosh (s/2)\sinh(t/2) \quad \text{and} \quad \p_tF=(1/2)\sinh (s/2)\cosh(t/2).
\] 
Since $s>0$ and $t>0$ we may assume $\p_sF\neq 0$ and $\p_tF\neq 0$ in the following discussion. 
By using the method of Lagrange multiplier, we deduce that if $(s,t)$ is the minimizer of the above function, then 
\begin{align}\label{note3}
\sinh\Big{(}\dfrac{s}{2}\Big{)}\sinh\Big{(}\dfrac{t}{2}\Big{)}=\frac{1}{2}, \qquad \dfrac{\tanh(s/2)}{\tanh(t/2)}=\dfrac{m_c t}{m_d s}.
\end{align}
Let us define
\[
H(s,t):=\frac{s\tanh(s/2)}{t\tanh(t/2)} \quad \text{and} \quad M:=\frac{m_c}{m_d}. 
\] 
For any given $s>0$, the first equation of \eqref{note3} determines 
\[
t(s):=2\sinh^{-1}\Big{(}\frac{1}{2\sinh(s/2)}\Big{)}.
\]
Thus, the problem is reduced to find a solution of the equation
\begin{align}\label{note4}
H(s, t(s))=M\quad {\rm for}\quad s>0\quad {\rm and}\quad M\in(0, \infty).
\end{align}
One can check that the function $H(s, t(s))$ is strictly increasing for $s\in (0,\infty)$ and the range coincides with $(0,\infty)$.
Hence there exists a unique solution $(s, t(s))$ of \eqref{note4} for any $M\in (0,\infty)$. 
This gives the unique minimizer of the energy functional $\Ec_\Cc$ for the weight $(m_d, m_c)$. 
If $M=1$, then $s=t=\log(2+\sqrt{3})$, namely, the corresponding hexagon is the regular hexagon.  If $M\to 0$ or $M \to \infty$, then $(s,t)\to (0,\infty)$ or $(s, t) \to (\infty, 0)$, respectively, i.e., if we change the ratio $M$ continuously from $1$ to $0$ or from $1$ to $\infty$,  
then the corresponding hexagon tends to an ideal triangle.

\subsection{Closed surfaces of genus greater than one}\label{Sec:example_genus_g}

Let $S_g$ be a closed oriented surface of genus $g \ge 2$.
We consider a graph $X$ and a continuous map $f:X \to S_g$ such that the complement of the image of $f$ gives a decomposition into $2(g-1)$-pairs of pants.
More precisely, for each $i=1, \dots, g-1$, let $\Hc_i$ be an oriented hexagon in the plane, $\wbar \Hc_i$ be a reflected copy of $\Hc_i$, and $\Pc_i=\Hc_i \cup \wbar \Hc_i$ be a pair of pants where we identify a set of every other sides of $\Hc_i$ with the corresponding set of sides of $\wbar \Hc_i$. 
Let $\Pc_i'$ be a homeomorphic copy of $\Pc_i$.
For $1 \le i \le g-1$, we glue $\Pc_i$ and $\Pc_i'$ by identifying two pairs of boundary components so that $\Pc_i \cup \Pc_i'$ is a compact connected oriented surface of genus $1$ with two boundary components, which we denote by $\d_i$, $\wbar \d_i$.
Then, we glue all pairs of pants by identifying $\wbar \d_i$ and $\d_{i+1}$ in the orientation-reversing way, where the indices $i$ are modulo $g-1$, and obtain a closed oriented surface $S_g$ of genus $g$.
Let $X$ be a finite graph whose edges are the (identified) sides of hexagons, and vertices are the (identified) corners of hexagons in $S_g$.
We denote by $X=(V, E, m_E)$ the corresponding weighted graph with unit weights $m_E \equiv1$, and by $f_g:X \to S_g$ the natural embedding.
Note that $f_g$ fills $S_g$ for every $g \ge 2$.
We shall find the unique minimizer of the energy functional $\Ec_{\Cc(g)}$ for $\Cc(g)=[f_g]$.

Actually, for any $g \ge 2$, we show that the unique minimizer is obtained by gluing $4(g-1)$ isometric copies of regular right-angled hexagons.
Consider an elements $\f_\ex$ of order $2$ and an element $\f_\rot$ of order $g-1$ in $\Mod(S_{g})$ such that $\f_\ex$ exchanges two hexagons $\Hc_i$ and $\wbar \Hc_i$ (resp.\ $\Hc_i'$ and $\wbar \Hc_i'$) in each $\Pc_i$ (resp.\ $\Pc_i'$) simultaneously in the orientation-preserving way, and $\f_{\rot}$ sends $\Pc_i \cup \Pc_i'$ and $\Pc_{i+1} \cup \Pc_{i+1}'$ for $i \mod g-1$.
Note that $\f_\ex, \f_\rot \in \Gc_{[f_g]}(S_{g})$.
If we have a hyperbolic surface $(S_{g}, G_{g, 0})$ realizing the unique minimizer of $\Ec_{\Cc(g)}$, then the image of the unique harmonic map in the class $\Cc(g)$ is invariant under the action of isometric group realizing $\Gc_{[f_g]}(S_g)$ by Theorem \ref{Thm:Nielsen}. 
This implies that $\Pc_i \cup \Pc_i'$ are isometric for all $i=1, \dots, g-1$, and two boundary components of $\Pc_{i}\cup \Pc_{i}'$ are closed geodesics which are isometric for each $i=1, \dots, g-1$.
Moreover, two hexagons $\Hc_i$ and $\wbar \Hc_i$ (resp.\ $\Hc_i'$ and $\wbar \Hc_i'$) in $\Pc_i$ (resp.\ $\Pc_i'$) are isometric for each $i=1, \dots, g-1$.
Since sides of hexagons are in the image of the harmonic map, the sides are geodesic arcs, and such arcs are contained in closed geodesics; since there are exactly four hexagons at each corner in the surface, each hexagon is right-angled.
Hence we may cut out $\Pc_{1}\cup \Pc_{1}'$ and glue two boundary components of $\Pc_{1}\cup \Pc_{1}'$ by orientation-reversing way.
The resulting surface $(S_{2}, G_{\ast})$ is a hyperbolic surface of genus $2$ endowed with the image of $f_{2}$.
Note that the total energy $\Ec_{\Cc(g)}$ is the sum of squared lengths of sides of hexagons and all $\Pc_i \cup \Pc_i'$ are isometric.
We observe that if $(S_g, G_{g, 0})$ is the unique minimizer of $\Ec_{\Cc(g)}$ if and only if $(S_2, G_{\ast})$ is the unique minimizer of $\Ec_{\Cc(2)}$.
Since the unique minimizer of $\Ec_{\Cc(2)}$ is given by gluing four isometric copies of regular right-angled hexagons in Section \ref{Sec:example_genus_2}, the hyperbolic surface $(S_g, G_{g, 0})$ is obtained by gluing $4(g-1)$ isometric copies of regular right-angled hexagons.

\bigskip

\begin{figure}[h!]
\centering
\includegraphics[width=145mm]{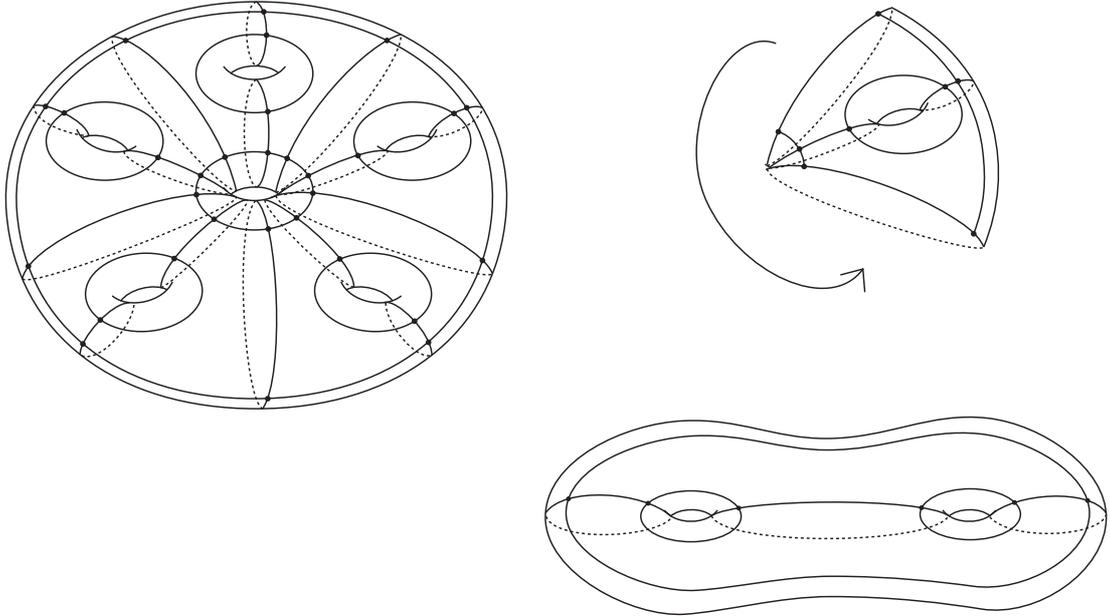}
\caption{The case of closed surface of genus $6$ with an embedded finite graph (the upper left), a fundamental domain of a rotation element of order $5$ in the mapping class group (the upper right), and the closed surface of genus $2$ with an embedded finite graph made from the fundamental domain above (the down right).}
\label{Fig:surface_g}
\end{figure}

\subsection{Closed surfaces associated with triangle tessellations}\label{Sec:example_triangle}

Let $\D$ be a geodesic triangle in $\H^2$ with inner angles $\pi/p$, $\pi/q$ and $\pi/r$ for positive integers $p$, $q$ and $r$ with $p^{-1}+q^{-1}+r^{-1}<1$.
We consider the associated triangle group $T(p, q, r)$, and a subgroup $\G$ of index $2$ with a fundamental domain $\D \cup \wbar \D$ (where $\wbar \D$ is a reflected copy of $\D$).
There exists a finite index, torsion-free normal subgroup $\G_0$ of $\G$
(\cite[Theorem in Section 8.6]{Stillwell}; in general, this follows from Selberg's lemma).
The quotient space $(S, G)=\G_0 \backslash \H^2$ is a closed hyperbolic surface.
The group $\G/\G_0$ acts on $(S, G)$ by isometry with a fundamental domain which is an isometric copy of $\D \cup \wbar \D$; this gives a triangle tessellation of $(S, G)$.
We consider the finite graph $X=(V, E, m_E)$ where edges arise as sides of triangles in $(S, G)$ and vertices are points of intersections of sides with unit weight $m_E \equiv 1$, and the natural embedding map $f:X \to S$.
Note that $f$ fills the surface $S$.
We claim that $(S, G)$ attains the minimum of $\Ec_\Cc$ with $\Cc=[f]$.

If $(S, G_0)$ attains the minimum of $\Ec_\Cc$, then the group $\Gc_{[f]}(S)$ is realized as an isometry group $\wt \Gc_0$ of $(S, G_0)$ by Theorem \ref{Thm:Nielsen}, and $\wt \Gc_0$ has to contain $\Gc_0$ which is an isomorphic copy of $\G/\G_0$.
The quotient space $\Gc_0\backslash (S, G_0)$ is a hyperbolic orbifold, which is obtained by gluing the sides of two copies of a triangle $\D'$ (Figure \ref{Fig:triangle}).
Since $\Gc_0$ is isomorphic to $\G/\G_0$ and is realized in $\wt \Gc_0$ via $\Gc_{[f]}(S)$, the triangle $\D'$ is isometric to $\D$.
\begin{figure}[h!]
\centering
\includegraphics[width=60mm]{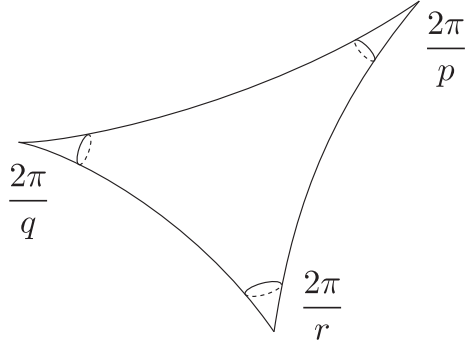}
\caption{A hyperbolic orbifold corresponding to the triangle of inner angles $\(\frac{\pi}{p}. \frac{\pi}{q}, \frac{\pi}{r}\)$.}
\label{Fig:triangle}
\end{figure}
Moreover, since $(\G/\G_0)\backslash(S, G)$ is isometric to this orbifold and $\Gc_0$ is isomorphic to $\G/\G_0$, the surface $(S, G_0)$ is isometric to $(S, G)$.
Therefore $(S, G)$ realizes the minimum of $\Ec_\Cc$ with $\Cc=[f]$.

Let $h:X \to (S, G)$ be the map such that the image coincides with that of $f$ and $h_e:[0, 1] \to (S, G)$ is the constant speed geodesic for each $e \in E$. 
The map $h$ is harmonic since it satisfies the balanced condition at each vertex.
Thus the energy is computed by
\[
\Ec_\Cc([G])=E_G(h)=|\G/\G_0|(l_1^2+l_2^2+l_3^2),
\]
where $l_1$, $l_2$ and $l_3$ are the side lengths of the triangle $\D$.

If we have a periodic weight function on the same graph, i.e., we assign weights $m_1$, $m_2$ and $m_3$ corresponding to the sides of $\D$ and extend periodically by $\G/\G_0$ on the surface, 
then $(S, G)$ is still the unique minimizer of $\Ec_\Cc$ with $\Cc=[f]$ and $h:X \to (S, G)$ is still harmonic.
The energy becomes
\[
\Ec_\Cc([G])=E_G(h)=|\G/\G_0|(m_1\, l_1^2+m_2\, l_2^2+m_3\, l_3^2),
\]
and $(S, G)$ is the unique minimizer of $\Ec_\Cc$ with $\Cc=[f]$ regardless of the values of $m_1$, $m_2$ and $m_3$.

\def\Klein{{\rm Klein}}

As a special case,
let us consider 
the triangle group $T(2, 3, 7)$.
For a subgroup $\G$ of $T(2, 3, 7)$ with index $2$,
there is a regular $14$-gon with all inner angle $2 \pi/7$ in the plane $\H^2$ as a fundamental domain of $\G_0$.
The quotient by this subgroup $\G_0$ is a hyperbolic surface known as the \textit{Klein quartic surface} (e.g., \cite[Example 2 in Section 7.3, p.176]{Stillwell}).
We denote the surface by $(S_3, G_{\Klein})$.
This is a surface of genus $3$ with the isometry group $\Gc_{0}$ of order $168$, which is the largest possible among hyperbolic surfaces of genus $3$, and such a surface is unique up to isometry.
(Recall that the order of isometry group for a closed hyperbolic surface of genus $g \ge 2$ is at most $84(g-1)$.)
Note that an isometric copy of a pair of triangles $\D$ and $\wbar \D$ is a fundamental domain of this isometry group $\Gc_{0}$ in $(S_3, G_{\Klein})$.
If we consider the graph $X$ and the map $f:X \to S_3$ given by sides of triangles in $(S_3, G_{\Klein})$ as above, 
the surface $(S_3, G_{\Klein})$ realizes the unique minimizer of the energy functional $\Ec_\Cc$ with $\Cc=[f]$.

In this case, the group $\Gc_{[f]}(S_3)$ is isomorphic to $\Gc_0$.
Indeed, first
the group $\Gc_{[f]}(S_3)$ has to contain an isomorphic copy of $\Gc_{0}$.
If $\Gc_{[f]}(S_3)$ is realized as an isometry group of some hyperbolic surface by Theorem \ref{Thm:Nielsen}, then that surface has to be isometric to $(S_3, G_{\Klein})$ since $\Gc_{0}$ and thus $\Gc_{[f]}(S_3)$ have the largest order of isometry group among hyperbolic surfaces of genus $3$, and this shows that $|\Gc_{[f]}(S_3)| = |\Gc_0|$; hence $\Gc_{[f]}(S_3)$ and $\Gc_{0}$ are isomorphic.
In particular, $(S_3, G_{\Klein})$ is the unique minimizer of $\Ec_\Cc$ with $\Cc=[f]$.

\appendix

\section{The first and second variation formulas}

We show the first and second variation formulae for finite weighted graphs into Riemannian manifolds.
This result is proved by Kotani and Sunada \cite[Theorem 2.1]{KotaniSunadaStandard}; we give a proof for the sake of convenience.

Let $X=(V, E, m_E)$ be a finite weighted graph and $(M, G)$  a Riemannian manifold.  
A map $f: X\to M$ is called  {\it piecewise $C^k$} if $f_{e}:=f|_e: [0,1]\to M$ is $C^k$ for each edge $e\in E$. We assume $k\geq 2$ throughout this appendix.
Consider a smooth family of piecewise $C^k$-maps  $f_{s}: X \to (M, G)$ for $s\in (-\epsilon, \epsilon)$ with $f:=f_{0}$ for $\e>0$.  
Let $f_e(s, t):=f_{s, e} (t)$ and $T_e(s, t):=(\p/\p t)f_e(s, t)$ for each edge $e\in E$. We denote the variation vector fields by $(\p/\p s)f_e(s, t)=V_e(s, t)$.
We denote the Levi-Civita connection and its curvature tensor of $G$ by $\nabla$ and $R$, respectively.

\begin{lemma}\label{Lem:first and second}
For any $s \in (-\e, \e)$, we have the first variation formula: 
\begin{align}\label{Eq:firstVF}
\frac{d}{d s} E_G(f_{s})=-2\sum_{e\in E}m_E(e)G( V_e, T_e)(s,0)-\sum_{e\in E}m_E(e)\int_0^1 G( V_e, \nabla_{T_e}T_e)(s,t) dt,
\end{align}
and the second variation formula:
\begin{align}\label{Eq:secondVF}
&\frac{d^2}{d s^2} E_G(f_{s})=-2\sum_{e\in E}m_E(e)G(\nabla_{V_e} V_e, T_e)(s,0)\nonumber\\
&+\sum_{e\in E}m_E(e) \int_0^1 \Big{\{}-G( \nabla_{V_e}V_e,  \nabla_{T_e}T_e)+\|\nabla_{T_e}V_e\|_G^2  -G(  R({V_e}, T_e)T_e, V_e) \Big{\}}(s, t)dt.
\end{align}
\end{lemma}

\begin{proof}
First, we see that
\begin{align}\label{f1}
\frac{d}{d s} E_G(f_{s})
&=\frac{1}{2}\sum_{e\in E}m_E(e)\int_0^1 \frac{\p}{\p s}G( T_e, T_e)\,dt
=\sum_{e\in E}m_E(e)\int_0^1 G( \nabla_{V_e}T_e, T_e)\, dt\nonumber\\
&=\sum_{e\in E}m_E(e)\int_0^1G(  \nabla_{T_e}V_e, T_e)\, dt \quad \quad \quad \text{(use $[V_e, T_e]=0$)}\nonumber\\
&=\sum_{e\in E}m_E(e)\int_0^1\frac{\p}{\p t}G(  V_e, T_e)- G(  V_e, \nabla_{T_e}T_e)\, dt\nonumber\\
&=\sum_{e\in E}m_E(e)\Big{[}G( V_e, T_e)\Big{]}_{t=0}^{t=1}-\sum_{e\in E}m_E(e)\int_0^1 G(  V_e,  \nabla_{T_e}T_e)\, dt.
\end{align}
Since 
\begin{equation}\label{Eq:ap1}
V_e(s, 1)=V_{\wbar e}(s, 0) \quad \text{and} \quad T_e(s,1)=- T_{\wbar e}(s,0)
\end{equation}
and $m_E(e)=m_E(\overline{e})$, we have that
\begin{align*}
\sum_{e\in E}m_E(e)G{(} V_e, T_e{)}(s,1)
	&=\sum_{e\in E}m_E(\overline{e})G{(} V_{\overline{e}}, -T_{\wbar e}{)}(s,0)\\
	&=-\sum_{e\in E}m_E(e)G{(} V_e, T_e{)}(s,0),
\end{align*}
and thus,
\begin{align}\label{f2}
\sum_{e\in E}m_E(e)\Big{[}G( V_e, T_e)\Big{]}_{t=0}^{t=1}=-2\sum_{e\in E}m_E(e)G( V_e, T_e)(s,0).
\end{align}
Combining \eqref{f2} with \eqref{f1}, we obtain the first variation formula \eqref{Eq:firstVF}.

Next, we have that
\begin{align}\label{f3}
\frac{d^2}{d s^2} E_G(f_{s})&=\frac{1}{2}\sum_{e\in E}m_E(e)\int_0^1 \frac{\p^2}{\p s^2}G( T_e, T_e)\, dt\nonumber\\
&=\sum_{e\in E}m_E(e)\int_0^1 G(\nabla_{V_e}\nabla_{V_e}T_e, T_e)+\|\nabla_{V_e}T_e\|^2\,dt\nonumber \\
&=\sum_{e\in E}m_E(e)\int_0^1 G(\nabla_{V_e}\nabla_{T_e}V_e, T_e)+\|\nabla_{T_e}V_e\|^2\,dt \quad (\text{use}\ [V_e,T_e]=0)\nonumber\\
&=\sum_{e\in E}m_E(e)\int_0^1 G(\nabla_{T_e}\nabla_{V_e}V_e, T_e)+G(R(V_e, T_e)V_e, T_e)+\|\nabla_{T_e}V_e\|^2\,dt.
\end{align}
The first term becomes 
\begin{align}\label{f4}
&\sum_{e\in E}m_E(e)\int_0^1 G(\nabla_{T_e}\nabla_{V_e}V_e, T_e)\, dt=\sum_{e\in E}m_E(e)\int_0^1 \frac{\p}{\p t}G(\nabla_{V_e}V_e, T_e)-G(\nabla_{V_e}V_e, \nabla_{T_e}T_e)\, dt \nonumber\\
&=-2\sum_{e\in E}m_E(e)G(\nabla_{V_e}V_e, T_e)(s,0)-\sum_{e\in E}m_E(e)\int_0^1G(\nabla_{V_e}V_e, \nabla_{T_e}T_e)\, dt,
\end{align}
where we used \eqref{f2} and  $\nabla_{V_e}V_e(s,1)=\nabla_{V_{\wbar e}}V_{\wbar e}(s,0)$. Substituting \eqref{f4} to \eqref{f3} and using the fact $G(R(V_e, T_e)V_e, T_e)=-G(R(V_e, T_e)T_e, V_e)$, we obtain the second variation formula \eqref{Eq:secondVF}.
\end{proof}

By the first variation formula \eqref{Eq:firstVF}, a map $f$ is a critical point of the energy functional $E_G$ if and only if 
\begin{align}\label{harmmap}
\nabla_{T_e}T_e=0\quad {\rm and}\quad \sum_{e\in E_x} m_E(e)T_e(x)=0
\end{align}
for any edge $e\in E$ and $x\in V$. 
We call the map satisfying \eqref{harmmap} a harmonic map.
Moreover, according to the second variation formula \eqref{Eq:secondVF}, we see the Hessian of $E_G$ for a harmonic map $h: X\to (M,G)$ is given by
\[
{\rm Hess}_{E_G}(V,W)_h=\sum_{e\in E} m_E(e)\int_0^1 \Big{\{}G( \nabla_{T_e}V_e, \nabla_{T_e}W_e ) -G(  R({V_e}, T_e)T_e, W_e )\Big{\}}\,dt,
\]
for any two variation vector fields $V, W$ of $h$. 
From the expression, we see that ${\rm Hess}_{E_G}$ is non-negative definite if $(M,G)$ has non-positive sectional curvature. If furthermore, $(M,G)$ is a compact Riemannian manifold of negative sectional curvature and the image of the harmonic map $h: X\to M$ is not homotopic to a point nor a closed circle, then ${\rm Hess}_{E_G}$ is non-degenerate at $h$ (see the proof of Theorem \ref{Thm:KSharmonic} (iii)).

\section{Smooth dependence of  harmonic maps}

\def\Met{{\rm Met}}
Let $(M,G)$ be a compact smooth Riemannian manifold. 
For any integer $k\geq 1$, we denote the set of all $C^k$-Riemannian metrics on $M$ by ${\rm Met}^k(M)$. Note that the space ${\rm Met}^k(M)$ is an open subset of the Banach space consisting of $C^k$-symmetric $(0,2)$-tensors on $M$.

\begin{proposition}\label{smooth dependence}
Fix any integer $k \ge 1$.
Let $X$ be a finite weighted graph, $(M, G)$ be a closed Riemannian manifold of nonpositive sectional curvature and $h: X\to M$ be a harmonic map. 
If the Hessian $\Hess_{E_G}$ of the energy functional $E_G$ is non-degenerate at $h$,
then there exists an open neighborhood $\Uc$ of $G$ in ${\rm Met}^{k+1}(G)$ and $C^k$-map $\wt h:\Uc \to C^{k+2}(X, M)$ such that 
$\wt h(G'):X \to M$ is a harmonic map for each $G' \in \Uc$ and $\wt h(G)=h$.
\end{proposition}

There is a corresponding result where the domain and the target are Riemannian manifolds by Eells-Lemaire \cite[Theorem 3.1]{EL} and Koiso \cite[Theorem 4.7]{Koiso}. 
We shall prove by a simple application of the implicit function theorem between Banach spaces as in \cite{EL}. 
Since we have not found the argument adapted to our setting, we describe the setup in a self-contained manner in the following.

For any piecewise $C^k$-map $f:X \to M$, 
let us denote the vector space consisting of $C^k$-vector fields along $f$ by 
$C^k(X, f^{-1}TM)$.
We understand that every $v$ in $C^k(X, f^{-1}TM)$ is in $C^k$ on each interior $(0, 1)$ of edges and continuous on $X$.
We define a norm on $C^k(X, f^{-1}TM)$ by
\[
\|v\|_k:=\sum_{0\leq i\leq k}\underset{e\in E_0}{\rm sup}\sup_{t \in [0, 1]}\|\nabla^i_{\p_t f_e} v_e\|_{G},
\]
where $E_0$ is the set of unoriented edges,
for each edge $e\in E_0$, $v_e$ is the restriction of $v$ and $\nabla^i_{\p_t f_e}$ denotes the $i$-th covariant derivative along $f_e$ relative to the metric $G$. 
Then, $(C^k(X, f^{-1}TM), \|\cdot\|_k)$ becomes a Banach space.

Let $C^k(X, M):=\{f:X\to M:  f\ \text{ is a piecewise $C^k$-map}\}$. Then, $C^k(X, M)$ is a  Banach manifold modeled on the Banach space $(C^k(X, f^{-1}TM), \|\cdot \|_k)$. Namely,  for any $f\in C^k(X, M)$ and an open neighborhood $B\subset C^k(X, f^{-1}TM)$ of $0$, we define the {\it exponential map} ${\rm exp}_f:B\to C^k(X,M)$ by 
$
({\rm exp}_fv)(x):={\rm exp}_{f(x)}v_{x},
$
where ${\rm exp}_{f(x)}: T_{f(x)}M\to M$ denotes the exponential map of the Riemannian manifold $(M,G)$. Then, for an open neighborhood $B$,  the map ${\rm exp}_f^{-1}: {\rm exp}_f(B)\to B$ gives a local chart  around the point $f$.

We consider the subset of $C^k(X, M)$ consisting of piecewise $C^k$-maps satisfying the balanced condition:
\[
C^k_{bal}(X,M):=\Big\{f\in C^k(X,M): \sum_{e\in E_x}m_E(e)\p_t f_e (x)=0 \ \ \text{for any $x\in V$}\Big\},
\]
where $\p_t f_e=(d/dt)f_e$.
Note that the set $C^k_{bal}(X,M)$ is not empty since there always exists a harmonic map, which satisfies the balanced condition.
For any $f \in C^k_{bal}(X,M)$, let us define the Banach space
\[
\mathcal{V}_{f, G}^k:=\Big\{v\in C^k(X, f^{-1}TM): \nabla^G_{v_x}\cdot (\sum_{e\in E_{x}}m_E(e) {\p_t{f}_e}(x))=0 \ \ \text{for any $x\in V$}\Big\},
\]
as a closed subspace of $C^k(X, f^{-1}TM)$.
Note that the condition in the definition $\mathcal{V}_{f, G}^k$
is equivalent to
\[
\sum_{e\in E_{x}}m_E(e) \nabla^G_{\p_t{f}_e}v_e(x)=0 \quad \text{for any $x\in V$}
\]
since $[v_e, \p_t{f}_e]=0$ for any $e \in E$.

\begin{lemma}
For any integer $k \ge 1$,
$C^k_{bal}(X,M)$ is a submanifold of $C^k(X,M)$ and for any $C^k$-metric $G$ on $M$ and any $f \in C^k_{bal}(X,M)$,
the tangent space of $C^k_{bal}(X,M)$ at $f$ is isomorphic to the Banach space $\mathcal{V}_{f, G}^k$.
\end{lemma}

\begin{proof}
Fix an arbitrary $f\in C^k_{bal}(X, M)$. Since $\mathcal{V}_{f, G}^k$ is a closed subspace of $C^k(X, f^{-1}TM)$, 
%we can induce the norm $\|\cdot\|_k$ to it so that $\mathcal{V}_{f,G}^k$ 
it becomes a Banach space with the induced norm $\|\cdot\|_k$.   
For a sufficiently small $v\in C^k(X, f^{-1}TM)$ with respect to $\|\cdot\|_k$, we set $f^v:={\rm exp}_f\, v\in C^k(X, M)$. First, we show $f^v\in C^k_{bal}(X, M)$ if $v\in \mathcal{V}_{f,G}^k$.
We take a geodesic normal coordinate $(U,\{x^1,\ldots, x^m\})$ of $(M,G)$ around a point $f(x)$ for $x\in V$, and define $ {\rm exp}: U\times \mathbb{R}^m\to M$  by 
$
{\rm exp}(p, v^1\ldots, v^m)={\rm exp}_{p}\Big(\sum_{i=1}^mv^i\frac{\p}{\p x^i}\Big|_{p}\Big).
$
By definition of ${\rm exp}$ and $f^v$, we have
 \[
 \frac{d}{d t} {f}^v_e(x)=\frac{d}{dt} {\rm exp}(f_e(t), v_e(t))\Big|_{t=0}=(d{\rm exp})_{(f(x), v_x)}\(\p_t f_e(x), \nabla^G_{\p_t{f}_e}v_e(x)\).
 \]
for each $e\in E_x$, and thus we obtain
\begin{align}\label{dexp}
\sum_{e\in E_x}m_E(e)\frac{d}{d t} f^v_e(x)=(d{\rm exp})_{(f(x), v_x)}\Big(0, \sum_{e\in E_x}m_E(e)\nabla^G_{\p_t{f}_e}v_e(x)\Big)
\end{align}
since $f\in C^k_{bal}(X, M)$.  Therefore, if $v\in \mathcal{V}_{f,G}^k$, then  $f^v\in C^k_{bal}(X, M)$ as required. 
 
Thus, we obtain a well-defined map 
\[
{\rm exp}_f:  \mathcal{U}\cap \mathcal{V}_{f,G}^k\to {\rm exp}_f(\mathcal{U})\cap C^k_{bal}(X,M)
\]
for an open neighborhood $\mathcal{U}$ in $C^k(X, f^{-1}TM)$ around $0$, and this map is injective if $\Uc$ is small enough since it is a restriction of the exponential map $\exp_f$ to $\mathcal{V}_{f, G}^k$.
We shall show the map is surjective, i.e., for any $f^v\in {\rm exp}_f(\mathcal{U})\cap C^k_{bal}(X, M)$, the corresponding vector field $v:={\rm exp}^{-1}_f(f^v)\in C^k(X, f^{-1}TM)$ is actually contained in $\mathcal{V}_{f, G}^k$. By \eqref{dexp},  it is sufficient to show that $(d{\rm exp})_{(x, v_x)}(0,w)=0$ implies $w=0$. Indeed, we have
\[
(d{\rm exp})_{(f(x), v_x)}(0,w)=d({\rm exp}_{f(x)})_{v_x}(w),
\]
and hence, if $\Uc$ is small enough so that for any $v \in \Uc$ satisfies that $\|v\|_k<\epsilon$ for sufficiently small $\epsilon>0$ (for instance, we may take $\epsilon$ as  ${\rm inj}(M)>0$),  then $d({\rm exp}_x)_{v_x}$ is injective, and $w=0$; as required. 
This implies the lemma.
 \end{proof}

We define the map
\[
\t: C_{bal}^{k+2}(X, M) \times \Met^{k+1}(M) \to \prod_{e \in E_0}\(C^k([0, 1], f_e^{-1}TM)\),
\]
\[
\t(f, G):=\(\nabla_{\p_t f_e}^G \p_t f_e\)_{e \in E_0},
\]
where $E_0$ is the set of unoriented edges of $X$; for each $e \in E_0$, we identify $e$ with $[0, 1]$ and denote by $C^k([0, 1], f_e^{-1}TM)$ the space of $C^k$-vector fields along $f_e:[0, 1] \to M$.

Let us fix a piecewise smooth harmonic map $h: X\to (M,G)$.
Taking a small enough open neighborhood $U_{bal}\subset \mathcal{V}_{h, G}^{k+2}$ of $0$, and
we identify ${\rm exp}_h(U_{bal})\subset C^{k+2}_{bal}(X, M)$ and $U_{bal}$.
For any $f \in U_{bal}$,
we identify
\[
\prod_{e \in E_0}\(C^k([0, 1], f_e^{-1}TM)\) \quad \text{and} \quad \prod_{e \in E_0}\(C^k([0, 1], h_e^{-1}TM)\)
\] 
via parallel transport relative to $\nabla^G$ along the curve $\gamma_f(t,x):={\rm exp}_{h(x)}(tv(x))$ where $f={\rm exp}_h\, v\in U_{bal}$.

We take any smooth curve $s \mapsto (f_s, G_s)$ in $U_{bal}\times {\rm Met}^{k+2}(M)$ through $(h, G)$ at $s=0$ for $s \in (-\e, \e)$ and for $\e>0$. 
Let us write $\t(f_s, G_s)=(\t_{s, e})_{e \in E_0}$ for $s \in (-\e, \e)$.
Then, by the equations \eqref{c4} and \eqref{c41} given in the proof of Lemma \ref{le3} (Note that the equations \eqref{c4} and \eqref{c41} hold for any Riemannian manifold as a target manifold), we see on each $e \in E_0$,
\begin{align}\label{derivtau}
\frac{d\tau_{s,e}}{ds}\Big{|}_{s=0}=\nabla_{T_e}\nabla_{T_e}V_e+R(V_e, T_e)T_e+\frac{d}{ds}\Big{|}_{s=0} \nabla^{G_s}_{T_e}T_e,
\end{align}
where $T_e:=\p_t h_e$ and $V_e:=(\partial_s f_{s, e})|_{s=0}$. 
Note that the final term in the right hand side depends only on $G_s$ and the initial condition $h$.

\begin{proof}[Proof of Proposition \ref{smooth dependence}] 
For a harmonic map $h:X \to (M, G)$,
we take an open neighborhood $U_{bal}$ of $h$ in $\mathcal{V}_{h, G}^{k+2}$ and an open neighborhood $\mathcal{U}\subset {\rm Met}^{k+1}(M)$ of $G$.
We regard $\tau$ as a $C^k$-map between Banach spaces:
\begin{align*}
\tau: U_{bal}\times \mathcal{U}\to \prod_{e \in E_0}\(C^k([0, 1], h_e^{-1}TM)\).
\end{align*}
We shall show that the derivative of $\t$ in the $U_{bal}$-component at $o=(h, G)$
\[
d\tau|_{U_{bal}, o}: \mathcal{V}_{h, G}^{k+2}\to \prod_{e \in E_0}\(C^k([0, 1], h_e^{-1}TM)\)
\]
is an  isomorphism. 

By \eqref{derivtau}, we have that $d\tau_e|_{U_{bal}, o}(V)=\nabla_{T_e}\nabla_{T_e}V_e+R(V_e, T_e)T_e$ on each $e \in E_0$ for any $V\in \mathcal{V}_{h, G}^{k+2}$ since the last term in \eqref{derivtau} is independent of the deformation $f_s$ of $h$ relative to $V$.
Thus, $d\tau|_{U_{bal}, o}$ is a bounded linear map.  
First we show that $d\tau|_{U_{bal}, o}$ is injective.
Indeed,
if $d\tau|_{U_{bal}, o}(V)=0$, then taking the deformation $f_s$ relative to $V$ in $C^{k+2}_{bal}(X, M)$ for any $s\in (-\e, \e)$ for $\e>0$, we see that
\begin{align*}
0&=\sum_{e\in E}m_E(e)\int_0^1 G(-\nabla_{T_e}\nabla_{T_e}V_e-R(V_e, T_e)T_e, V_e)\, dt 
+\sum_{e \in E}m_E(e)G_0(\nabla_{T_e}V_e, V_e)\Big|_{t=0}^{t=1}\\
&={\rm Hess}_{E_G}(V,V)_h,
\end{align*}
where we used Lemma \ref{le2} in the first equality which follows from the assumption $f_s\in C^{k+2}_{bal}(X,M)$.
The assumption in the statement implies that ${\rm Hess}_{E_G}$ is non-degenerate at $h$, we obtain $V=0$, and thus $d\tau|_{U_{bal}, o}$ is injective.  

Next we show that $d\tau|_{U_{bal}, o}$ is surjective.
For any given $W=(W_e)_{e \in E_0}$ where $W_e \in C^k([0, 1], h_e^{-1}TM)$ for each $e \in E_0$, we consider the second order ordinary differential equations
\begin{align}\label{dtauw}
\nabla_{T_e}\nabla_{T_e}V_e+R(V_e, T_e)T_e=W_e
\end{align}
on each $e \in E_0$ with the condition that $V$ is in $\mathcal{V}_{h, G}^{k+2}$, i.e., $V$ solves (\ref{dtauw}) on each interior $(0, 1)$ of edges $e$, and is continuous on $X$
satisfying the balanced condition. 
Let us consider the Hilbert space $H^{1, 2}_{h, G}$ as the completion of $\mathcal{V}_{h, G}^{k+2}$ endowed with the inner product
\[
\langle V^1, V^2\rangle_{H_{h, G}^{1, 2}}:=\sum_{e \in E}m_E(e)\int_0^1\Big{\{}G(\nabla_{T_e}V_e^1, \nabla_{T_e}V_e^2)-G(R(V_e^1, T_e)T_e, V_e^2)\Big{\}}\,dt,
\]
where this indeed defines the positive definite inner product since $(M, G)$ has nonpositive sectional curvature and ${\rm Hess}_{E_G}$ is non-degenerate at $h$.
Note that every $V$ in $H^{1, 2}_{h, G}$ is continuous on $X$.
For any $W \in \prod_{e \in E_0}\(C^k([0, 1], h_e^{-1}TM)\)$, we define a linear functional $L_W$ on $H^{1, 2}_{h, G}$ by
\[
L_W(\f):=\sum_{e \in E}m_E(e)\int_0^1 G(W_e, \f_e)\,dt, \quad \text{for $\f \in H^{1, 2}_{h, G}$}.
\]
Then, $L_W$ is a bounded linear functional on $H^{1, 2}_{h, G}$; in fact, the natural embedding map $H^{1, 2}_{h, G} \to C^0(X, h^{-1}TM)$ is compact.
Therefore the Riesz representation theorem implies that there exists a unique $V$ in $H^{1, 2}_{h, G}$ such that
$\langle V, \f\rangle_{H_{h, G}^{1, 2}}=-L_W(\f)$ for any $\f$ in $H^{1, 2}_{h, G}$.
This implies that $V$ solves (\ref{dtauw}) on each interior $(0, 1)$ of edges $e$ and $V_e$ is in $C^{k+2}$ on $(0, 1)$ by taking smooth functions $\f_e$ whose supports are included in $(0, 1)$ for each $e \in E_0$.
Then, integration by parts gives
\begin{align*}
\sum_{e \in E}m_E(e)\int_0^1\Big{\{}G(-\nabla_{T_e}\nabla_{T_e}V_e-R(V_e, T_e)T_e, \f_e)\Big{\}}\,dt &-2\sum_{x \in V}\sum_{e \in E_x}G(\nabla_{T_e}V_e, \f_e)\\
&=-\sum_{e \in E}m_E(e)\int_0^1 G(W_e, \f_e)\,dt,
\end{align*}
and
$\sum_{x \in V}\sum_{e \in E_x}G(\nabla_{T_e}V_e, \f_e)=0$ holds for any $\f$ in $H^{1, 2}_{h, G}$.
Therefore $V$ satisfies the balanced condition and is in $V_{h, G}^{k+2}$.
Hence $d\tau|_{U_{bal}, o}$ is surjective.

Now, the implicit function theorem for Banach spaces implies that there exist an open neighborhood $\mathcal{U}'\subset \mathcal{U}$ of $G$ and a unique $C^k$-map 
$\wt h: \mathcal{U}'\to U_{bal}$ satisfying that
\[
\tau(\wt h(G'),G')=0
\]
on $\mathcal{U}'$. Then, $\wt h(G): X\to M$ is a harmonic map relative to the metric $G'$ for any $G' \in \Uc'$ and $\wt h(G)=h$, and we complete the proof. 
\end{proof}

\subsection*{Acknowledgements} 
 The authors would like to thank Professor Sumio Yamada for suggesting us the problem and a number of helpful and inspiring discussions, Professor Motoko Kotani for helpful discussions concerning on the standard realizations in an early stage of this project and for providing opportunities that made this collaboration possible, and Professors Jayadev Athreya and Toshiyuki Sugawa for useful comments. 
 T.K.\ is supported by JSPS KAKENHI Grant Number JP18K13420, and 
 R.T.\ is supported by JSPS KAKENHI Grant Numbers JP17K14178 and JP20K03602.

\bibliographystyle{alpha}
\bibliography{harmonic_map}

\end{document}